\documentclass[11pt,reqno]{amsart}
\usepackage[T1]{fontenc}
\usepackage[utf8]{inputenc}
\usepackage{lmodern}
\usepackage[a4paper,textwidth=158mm,textheight=235mm,centering]{geometry}
\usepackage{amsmath,amssymb,mathtools,mathrsfs,booktabs}
\usepackage[protrusion=true,expansion=false]{microtype}
\usepackage[hidelinks,pdfusetitle]{hyperref}
\usepackage{bookmark}
\hypersetup{pdftitle={Minimality and Nonlinear Spectral Rigidity of Hypersurfaces in Real Space Forms},pdfsubject={Euclidean, spherical and hyperbolic hypersurfaces with nonlinear mean-curvature response},pdfkeywords={real space forms, hyperbolic hypersurfaces, nonlinear spectral rigidity, Chen's conjecture, analytic continuation, local normal-graph rigidity}}
\numberwithin{equation}{section}
\newtheorem{theorem}{Theorem}[section]
\newtheorem{proposition}[theorem]{Proposition}
\newtheorem{lemma}[theorem]{Lemma}
\newtheorem{corollary}[theorem]{Corollary}
\newtheorem{standard}[theorem]{Standard theorem}
\theoremstyle{definition}\newtheorem{definition}[theorem]{Definition}
\theoremstyle{remark}\newtheorem{remark}[theorem]{Remark}
\newcommand{\R}{\mathbb R}
\newcommand{\Sph}{\mathbb S}\newcommand{\Hyp}{\mathbb H}
\newcommand{\ip}[2]{\langle #1,#2\rangle}
\newcommand{\norm}[1]{\lVert #1\rVert}
\newcommand{\tr}{\operatorname{tr}}\newcommand{\Div}{\operatorname{div}}
\newcommand{\Hess}{\operatorname{Hess}}\newcommand{\Ric}{\operatorname{Ric}}
\newcommand{\dist}{\operatorname{dist}}\newcommand{\Id}{\operatorname{Id}}
\newcommand{\sgn}{\operatorname{sgn}}\newcommand{\dd}{\,d}
\newcommand{\JJ}{\mathscr J}\newcommand{\TT}{\mathcal T}
\newcommand{\KK}{\mathcal K}\newcommand{\RR}{\mathcal R}
\newcommand{\arxiv}[1]{\href{https://arxiv.org/abs/#1}{arXiv:#1}}
\newcommand{\doi}[1]{\href{https://doi.org/#1}{doi:#1}}
\allowdisplaybreaks[2]
\title[Nonlinear spectral rigidity in real space forms]{Minimality and Nonlinear Spectral Rigidity of Hypersurfaces in Real Space Forms}

\author{Zihao Wang}
\address{School of Mathematical Sciences, Fudan University, Shanghai 200433, China}
\email{wangzh25@m.fudan.edu.cn}
\date{September 22, 2026}

\AtBeginDocument{\hypersetup{pdfauthor={\authors}}}
\begin{document}
\raggedbottom
\begin{abstract}
We study nonlinear mean-curvature spectral equations for hypersurfaces in Euclidean, spherical and hyperbolic space. For an admissible analytic response $F$ and ambient sectional curvature $c$, we prove minimality when the spectral parameter satisfies $\sigma\ge nc$, and strict mean-curvature bounds at gradient points when $\sigma<nc$. The class includes every power $F(t)=t^{p-1}$, $p>1$, and positive mixtures satisfying an explicit width condition. In particular, the result proves Chen's conjecture in Euclidean codimension one and biharmonic minimality in constant negative curvature, without completeness or holonomicity. A common real-analytic continuation and focal-trace argument treats all three geometries. Explicit umbilical models show that the unconditional spectral minimality range is sharp. Independently, we prove quantitative local rigidity, modulo ambient isometries, of compact nonminimal geodesic spheres in all three geometries and proper equal-radius Clifford hypersurfaces in spheres, for smooth laws with $F,F'>0$ near the model amplitude.
\end{abstract}
\maketitle
\tableofcontents

\section{Introduction and main results}\label{sec:intro}
The ambient models are the three simply connected real space forms
\begin{equation}\label{eq:three-models}
 \mathbb M_c^{n+1}=
 \begin{cases}
 \R^{n+1},&c=0,\\
 \Sph^{n+1}(c^{-1/2}),&c>0,\\
 \Hyp^{n+1}(-a^2),&c=-a^2<0,\quad a>0.
 \end{cases}
\end{equation}
We formulate the same response equation and the same families of results in all three models. The local arguments also apply to any Riemannian manifold $N^{n+1}(c)$ of constant sectional curvature $c$, including model open subsets and quotients; only the compact-model deformation theorem uses an entire model in $\mathbb M_c^{n+1}$.

Unless stated otherwise, manifolds have no boundary. Let $X:M^n\to N^{n+1}(c)$ be a smooth hypersurface immersion, with induced metric $g=X^*\bar g$. Choose a local unit normal and use
\[
 A=-\bar\nabla\nu,\qquad \tau=\tr A=nH,\qquad
 \TT(X)=\tau\nu,\qquad\Delta=\Div\nabla.
\]
Our sign convention for the curvature tensor is
$R(U,V)W=\nabla_U\nabla_VW-\nabla_V\nabla_UW-\nabla_{[U,V]}W$.
For a vector field $V$ along $X$, define the explicitly signed operator
\begin{equation}\label{eq:J-intro}
 \JJ_XV=\Delta^X V-\sum_{a=1}^n\bar R(dX(e_a),V)dX(e_a).
\end{equation}
Thus $\JJ_X\TT(X)=0$ is the standard biharmonic-map equation. The domain metric is fixed during map variation, but for each immersion under consideration it is its own induced metric. This distinction is essential when normal graphs are compared.

Given $F:(0,\infty)\to(0,\infty)$, write
\begin{equation}\label{eq:Phi}
 \Phi(t)=\sgn(t)F(|t|),\qquad
 W_F(X)=\Phi(\tau)\nu\quad(\tau\ne0).
\end{equation}
The vector field is independent of the local choice of normal. We investigate
\begin{equation}\label{eq:spectral-main}
 \JJ_XW_F(X)=\sigma W_F(X),\qquad \sigma\in\R,
\end{equation}
classically on the nonzero-mean-curvature set. This convention avoids imposing smoothness on a fractional power at its zero set. When $F(t)=t^{p-1}$, equation \eqref{eq:spectral-main} becomes
\begin{equation}\label{eq:power-main}
 \JJ_X\bigl(|\TT|^{p-2}\TT\bigr)=\sigma|\TT|^{p-2}\TT,
 \qquad p>1.
\end{equation}
On the set $\TT\ne0$, the case $\sigma=0$ is the Euler equation of the fixed-domain $p$-bienergy $p^{-1}\int|\TT|^p$. The normalization of $\TT=nH\nu$ is irrelevant for the power equation. We do not use the alternative convention involving the $p$-tension field; see \cite{HL,TCL} for this distinction.

\subsection{Analytic constitutive laws}
Put $\Lambda=nc-\sigma$. We first state the result for analytic response laws; power laws and their Euclidean and hyperbolic consequences follow immediately afterward.
\begin{definition}\label{def:admissible}
An \emph{admissible constitutive law} is a positive real-analytic function $F$ on $(0,\infty)$ whose logarithmic derivative
\begin{equation}\label{eq:elasticity}
 R_F(t)=\frac{tF'(t)}{F(t)}
\end{equation}
satisfies, for some finite constants $0<r_-\le r_+$,
\begin{equation}\label{eq:law-hypotheses}
 r_-\le R_F(t)\le r_+,
 \qquad R_F(t)-tR_F'(t)>0\quad(t>0).
\end{equation}
\end{definition}
Define
\begin{equation}\label{eq:k-of-t}
 k_F(t)=\frac{F(t)}{2F'(t)}=\frac{t}{2R_F(t)}.
\end{equation}
It is an increasing analytic diffeomorphism of $(0,\infty)$ onto itself. For $n\ge2$ and $\Lambda>0$, let $t_F(\Lambda)$ be the unique positive solution of
\begin{equation}\label{eq:implicit-threshold}
 k_F(t)^2+\frac{(t+k_F(t))^2}{n-1}=\Lambda.
\end{equation}
Existence and uniqueness follow from the strict increase and the limits at zero and infinity of $k_F$.
Here constant mean curvature (CMC) on a possibly nonorientable hypersurface means constancy of the norm of the mean-curvature vector; a nonzero constant vector norm canonically supplies a normal orientation.
\begin{theorem}[Nonlinear spectral rigidity]\label{thm:nonlinear}
Let $F$ be admissible and let $X$ satisfy \eqref{eq:spectral-main}. If $\sigma\ge nc$, then $X$ is minimal. If $\Lambda=nc-\sigma>0$ and $n\ge2$, then
\[
 |H|<t_F(\Lambda)/n
\]
at every point where $dH\ne0$. A connected non-CMC solution satisfies the corresponding non-strict bound everywhere; a strict exceedance at one point forces CMC on that component. On a nonminimal CMC component one has $|A|^2=\Lambda$ and $|H|\le\sqrt{\Lambda/n}$. For $n=1$, the curvature is locally constant.
\end{theorem}
The power specialization is stated in Theorem~\ref{thm:power} below and proved in Section~\ref{sec:applications}. An explicit nonhomogeneous class is obtained as follows.
\begin{corollary}[Positive mixtures of powers]\label{cor:mixtures}
Let
\begin{equation}\label{eq:mixture}
 F(t)=\sum_{j=1}^m a_jt^{r_j},\qquad a_j>0,
 \quad r_-:=\min_jr_j>0,
 \quad (\max_jr_j-r_-)^2<4r_-.
\end{equation}
Then $F$ is admissible, so Theorem~\ref{thm:nonlinear} applies. In particular the equation with $\Phi(\tau)=\tau+|\tau|\tau$ has the same minimality conclusion whenever $\sigma\ge nc$.
\end{corollary}
The sufficient width condition in \eqref{eq:mixture} follows from the variance estimate in Section~\ref{sec:applications}.

\subsection{Power laws in all three space forms}
Taking $p=2$ and $\sigma=0$ in the following power-law statement gives the Euclidean and hyperbolic biharmonic assertions when $c\le0$.
\begin{theorem}[Power-law spectral rigidity]\label{thm:power}
Let $n\ge1$, $c,\sigma\in\R$ and $p>1$. A smooth hypersurface satisfying \eqref{eq:power-main} is minimal if $\Lambda\le0$. If $\Lambda>0$ and $n\ge2$, then at every point with $dH\ne0$,
\begin{equation}\label{eq:explicit-threshold}
 |H|<h_{n,p}(\Lambda):=
 \frac{2(p-1)}{n}
 \sqrt{\frac{(n-1)\Lambda}{n-1+(2p-1)^2}}.
\end{equation}
A connected non-CMC hypersurface satisfies the corresponding non-strict bound everywhere. A strict exceedance of this bound at one point therefore forces CMC on the connected component. On a nonminimal CMC component one necessarily has
\begin{equation}\label{eq:CMC-norm}
 |A|^2=\Lambda,\qquad |H|\le\sqrt{\Lambda/n}.
\end{equation}
For $n=1$ the curvature is locally constant. No completeness, properness, orientability, embeddedness or condition on transverse principal-curvature multiplicities is assumed.
\end{theorem}

\begin{corollary}[Biharmonic minimality in nonpositive space forms]\label{cor:euclidean}
For every integer $n\ge1$, every smooth biharmonic hypersurface immersion into a Riemannian manifold of constant sectional curvature $c\le0$, equipped with its induced metric, is minimal. In particular, this holds in $\R^{n+1}$ and in $\Hyp^{n+1}(-a^2)$ for every $a>0$, without completeness or holonomicity assumptions.
\end{corollary}
\begin{proof}
Set $p=2$ and $\sigma=0$ in Theorem~\ref{thm:power}. Then $\Lambda=nc\le0$, in both the flat and the negative-curvature cases.
\end{proof}
The Euclidean specialization is the codimension-one assertion in Chen's problem \cite{Chen91}; the statement concerns arbitrary hypersurface immersions rather than only complete or embedded ones.

For the ordinary unit-sphere equation, \eqref{eq:explicit-threshold} reduces to
\[
 h_{n,2}(n)=2\sqrt{\frac{n-1}{n(n+8)}}.
\]
The threshold gives a sufficient one-point CMC criterion. Its relation to classification questions is discussed in Section~\ref{sec:scope}.

\subsection{The three geometries and the spectral boundary}\label{subsec:three-regimes}
Theorem~\ref{thm:nonlinear} applies to each row of the following table, with the same admissibility hypotheses. Here $a>0$ in the hyperbolic row.
\begin{center}
\small
\renewcommand{\arraystretch}{1.18}
\begin{tabular}{@{}lll@{}}
\toprule
Ambient model & Minimality range & Parameter below that range \\
\midrule
$\R^{n+1}$ & $\sigma\ge0$ & $\Lambda=-\sigma>0$ \\
$\Sph^{n+1}(c^{-1/2})$, $c>0$ & $\sigma\ge nc$ & $\Lambda=nc-\sigma>0$ \\
$\Hyp^{n+1}(-a^2)$ & $\sigma\ge-na^2$ & $\Lambda=-na^2-\sigma>0$ \\
\bottomrule
\end{tabular}
\end{center}
In every subcritical row, the gradient-point bound is $|H|<t_F(\Lambda)/n$; the non-CMC global bound, the one-point CMC criterion and $|A|^2=\Lambda$ on nonminimal CMC components also apply. For powers, $t_F(\Lambda)/n$ is the explicit expression in \eqref{eq:explicit-threshold}. Section~\ref{sec:applications} spells out the hyperbolic formulas, including nonhomogeneous mixtures.

\begin{proposition}[Sharp unconditional spectral boundary]\label{prop:sharp-boundary}
For every $n\ge1$, $c\in\R$ and $\sigma<nc$, there exists a complete embedded nonminimal totally umbilical hypersurface in $\mathbb M_c^{n+1}$ satisfying \eqref{eq:spectral-main}, for any smooth positive law defined near its mean-curvature amplitude. Consequently, for admissible laws the range $\sigma\ge nc$ is the exact range in which all hypersurface solutions must be minimal, when no compactness or other additional hypothesis is imposed.
\end{proposition}
The examples and their direct computations are given in Section~\ref{subsec:umbilical-range}. This assertion concerns the boundary for unconditional minimality, not the optimality of the non-CMC mean-curvature bound. In the hyperbolic case some examples are noncompact; they are not covered by the compact-graph theorem below.

\subsection{Quantitative local rigidity of compact models}
The next theorem is independent of the maximal-endpoint argument. Its constitutive law needs only be smooth near the nonzero model mean curvature, not analytic or globally admissible.
\begin{theorem}[Rigidity for increasing laws]\label{thm:local}
Let $F$ be smooth and satisfy $F>0$, $F'>0$ on a neighborhood of $|\tau_0|$. Let $M_0$ be either of the following entire compact models in the simply connected real space form:

\smallskip\noindent
(a) A nonminimal geodesic sphere with $A_0=\lambda\Id$, $\lambda\ne0$, and $K_0=\lambda^2+c>0$, equipped with the parameter
\[
 \sigma=n(c-\lambda^2).
\]
(b) In $\Sph^{n+1}(c^{-1/2})$, $c>0$, a product
\[
 \Sph^a((2c)^{-1/2})\times\Sph^b((2c)^{-1/2}),
 \qquad a,b\ge1,\quad a\ne b,\quad a+b=n,
\]
with $\sigma=0$.

\smallskip\noindent
There are $\varepsilon,C>0$ such that a smooth global normal graph over $M_0$ with $C^4$ norm less than $\varepsilon$ can be moved by a small ambient isometry to a graph $v$ satisfying
\begin{equation}\label{eq:quantitative-main}
 \norm{v}_{H^4(M_0)}
 \le C\left(\norm{\mathcal F_N(v)}_{L^2}
              +\norm{\mathcal F_T(v)}_{L^2}\right).
\end{equation}
Here $\mathcal F_N,\mathcal F_T$ are the normal and tangential defects of \eqref{eq:spectral-main}, pulled back to $M_0$ as in Section~\ref{sec:local}. In particular, every such graph solving \eqref{eq:spectral-main} is congruent to $M_0$. For powers, the neighborhood and estimate can be chosen uniformly when $p$ ranges over a fixed compact subinterval of $(1,\infty)$.
\end{theorem}
The normal components of ambient Killing fields form the exact infinitesimal kernel, so the isometry in the conclusion is essential. Part (a) treats geodesic spheres in all three models. In $\Hyp^{n+1}(-a^2)$, a sphere of geodesic radius $R>0$ has
\begin{equation}\label{eq:hyperbolic-sphere-intro}
 \lambda=a\coth(aR),\qquad
 K_0=\frac{a^2}{\sinh^2(aR)},\qquad
 \sigma=-na^2\bigl(1+\coth^2(aR)\bigr)<-2na^2.
\end{equation}
These hyperbolic models are therefore not solutions of the ordinary biharmonic equation $\sigma=0$. Part (b), which is specific to the spherical ambient geometry, treats every positive logarithmic derivative, including the values at which a tangential coefficient degenerates. Horospheres and equidistant hyperbolic hypersurfaces are useful spectral examples but are not added to this compact-model rigidity statement.

\subsection{Previous work and the precise comparison}
The ordinary biharmonic variational theory originates in Eells--Lemaire and Jiang \cite{EL,Jiang,JiangV}. Chen's Euclidean problem concerns all codimensions \cite{Chen91}; the Euclidean case of Corollary~\ref{cor:euclidean} addresses its hypersurface case. For broader histories and references, see \cite{OC,FO22,FYZ22,Chen25}.

Early Euclidean hypersurface results include Dimitri\'c \cite{Dimitric} and the ambient-$\R^4$ work of Hasanis--Vlachos and Defever \cite{HV,Defever}. Fu--Hong--Zhan establish the ambient-$\R^5$ and ambient-$\R^6$ cases \cite{FHZ5,FHZ6}, corresponding to hypersurface dimensions $n=4$ and $n=5$. Restricted principal-curvature, scalar-curvature, ideal and weak-convexity classes are treated in \cite{Fu3,FH,CM13,DA,Luo}. In nonzero constant curvature, four-dimensional hypersurfaces are studied in \cite{GLV}. For Fu's spherical three-principal-curvature result \cite{FuS}, the exceptional elimination case was subsequently repaired by Andronic--Fu--Oniciuc \cite{AFO}; that repair concerns the spherical paper, rather than the Euclidean paper \cite{Fu3}.

Koiso--Urakawa \cite[Proposition~4.7]{KU} obtain the curvature-pair ODE used below; their nonexistence theorem imposes additional principal-curvature and frame hypotheses. Montaldo--Oniciuc--Ratto \cite{MOR} treat cohomogeneity-one hypersurfaces. The unrestricted claim in the withdrawn preprint \cite{SYwithdrawn} is not used; the withdrawal notice reports missing proof steps.

Bibi--Soret--Ville \cite[Theorems~1 and~2]{BSV} prove Euclidean biharmonic minimality under the holonomicity hypothesis, meaning that the principal curvature directions locally form an integrable coordinate net. Their proof uses the holonomic isoparametric structure of the mean-curvature level sets. The Euclidean specialization of Corollary~\ref{cor:euclidean} removes that hypothesis from their minimality conclusion; it does not extend their separate construction and classification of biconservative hypersurfaces. The argument here derives a commuting transverse system along a single gradient curve in Section~\ref{sec:geometry}, without a principal-coordinate net on a neighborhood. Sections~\ref{sec:continuation} and~\ref{sec:endpoint} analytically continue this necessary system and exclude every possible forward endpoint using positive focal factors and the complete active trace. The nonlinear spectral results in other space forms and the independent compact-model rigidity theorem go beyond this Euclidean comparison.

Properness and integral-growth methods provide other approaches \cite{AM,Maeta14,NUG14,Wheeler}. In particular, \cite{NUG14} assumes a complete source together with finite energy and bienergy, while the cited rigidity result of Wheeler is in the positive-$\varepsilon$ regime. Recent results also concern distinct settings: the hypersurface evolution studied in \cite{FHT}, simple-rotational surfaces of codimension two in \cite{Maeta26}, and free-boundary biharmonic hypersurfaces in the Euclidean unit ball in \cite{SY26}. In the last setting, Seo--Yun obtain an isoperimetric inequality and minimality under additional boundary conditions on the mean curvature. These hypotheses are separate from the local gradient-line argument used here.

The standard spherical examples and their classifications occur in \cite{CMO1,CMO2,BMO08,BMO}; CMC gaps and continuation results are treated in \cite{Nistor21,BLO}. In particular, the $p=2$ spectral equation and its umbilical examples have been studied by Andronic--Nistor \cite[Corollary~3.2 and Proposition~3.4]{AN24}. Their conventions and ours are related by
\[
 \Delta_{\mathrm{AN}}=-\Delta,\qquad
 \lambda_{\mathrm{AN}}=\sigma,\qquad
 |A|^2=nc-\sigma\quad\hbox{on a nonminimal CMC component}.
\]
The CMC identity, its Cauchy--Schwarz bound and the standard umbilical examples are background facts. The estimate in \eqref{eq:explicit-threshold} instead concerns gradient points; in particular, its hypotheses differ from the completeness and CMC assumptions of the gap theorem in \cite{Nistor21}.

There is a qualification concerning the spectral threshold in \cite{AN24}. In arXiv:2403.10155v1, Remark~3.7 states minimality at that threshold without repeating the CMC hypothesis, whereas Theorem~3.6 explicitly assumes CMC and uses it in deriving equation~(3.9). Our comparison with that proof is restricted to its CMC setting. The non-CMC assertion in the present paper follows from Theorem~\ref{thm:endpoint}, independently of the wording of the remark.

The compact-model result also concerns a different operator from fixed-domain bienergy second variation. Ou \cite{Ou20} computes stability indices of biharmonic hypersurfaces. Montaldo--Oniciuc--Ratto \cite{MOR22} determine the index, nullity and Jacobi kernel for the biharmonic Clifford torus $\Sph^1(1/2)\times\Sph^1(1/2)\to\Sph^4$, an immersion of codimension two. Here the objects are hypersurfaces, every graph carries its newly induced metric, and both the normal and tangential equations are linearized. Section~\ref{sec:local} proves the joint kernel statement for every positive model elasticity; Section~\ref{sec:nonlinear-local} derives the nonlinear estimate modulo ambient isometries. This is the precise deformation problem addressed by Theorem~\ref{thm:local}.

The term ``local rigidity'' also occurs in the CMC classification problem of
Chen--Li \cite{ChenLi25} and Cheng--Chen--Li \cite{ChengChenLi26}. Their results
assume CMC together with additional constant higher-order curvature quantities
and restrictions on the principal curvatures; in particular, the first paper
treats at most three distinct principal curvatures. These are not assumptions
on the perturbed graphs in Theorem~\ref{thm:local}. Small graph norm alone does
not preserve principal-curvature multiplicities: an eigenspace of a model may
split under a perturbation. Nor is constant mean curvature assumed for those
graphs. Our comparison is therefore between different hypothesis sets, not a
claim that all CMC rigidity statements are new. The contributions asserted here
are the unrestricted endpoint exclusion for the indicated nonlinear spectral
system and the quantitative, induced-metric graph estimate for the specified
compact models. The curvature-pair reduction, the ordinary spectral equation,
the CMC norm identity and the standard model examples have the antecedents
listed above.

Nonlinear bienergies and ordinary-tension $p$-bienergies have antecedents in \cite{HFe,Han15,HZ,CL,HL}, and small-curvature-integral rigidity is studied in \cite{BDST}. Their global integral or convexity hypotheses differ from the response hypotheses used here. The functional is an existing notion, and our convention is
\[
 \frac1p\int|\TT(X)|^p=E_{2,p}(X)
 \quad\hbox{in the two-exponent notation of \cite{TCL}}.
\]
The required hypersurface equations are derived in Section~\ref{sec:geometry}.

The endpoint construction is related to the osculating-circle geometry of the Tait--Kneser theorem \cite{GTT,BJT}. The distance inequalities and terminal estimates needed here are proved directly. Section~\ref{sec:scope} records the roles of constant sectional curvature, codimension one and the graph hypothesis.

\subsection{Notation and the two proof chains}
The following scalars have different roles:
\[
\begin{array}{c|l}
\tau=nH&\text{signed trace of the hypersurface shape operator}\\
\Lambda=nc-\sigma&\text{normal potential parameter}\\
t=|\tau|,
\ \kappa=k_F(t)&\text{mean-curvature amplitude and gradient principal curvature}\\
\lambda_i,\mu_i&\text{transverse entries of }A\text{ and }-\nabla E\\
\Theta(\kappa),Y(\kappa)&\text{trace response and normal response}\\
j_i,J=\prod j_i&\text{positive focal factors and their product}.
\end{array}
\]
To avoid reusing the mean-curvature amplitude $t$, Section~\ref{sec:endpoint} uses $\zeta=b-s$ for remaining time and $\rho$ for the osculating radius. Its scalar $\mathsf B=\langle N,N_*\rangle$ is distinct from the transverse operator $B$. In Section~\ref{sec:local}, $r_0$ is the response elasticity at the model value, $\psi=\dot\tau$ is the trace variation, and $v$ is reserved for a sliced normal graph. The global-in-dimension statements have the dependency chain
\begin{gather*}
 \eqref{eq:two-equations}\ \Longrightarrow\
 \eqref{eq:pairs},\eqref{eq:full-ode}\ \Longrightarrow\
 \text{analytic persistence},\\
 \text{analytic persistence}\ \Longrightarrow\
 \text{Theorem~\ref{thm:endpoint}}\ \Longrightarrow\
 \text{Theorems~\ref{thm:power},\ref{thm:nonlinear}}.
\end{gather*}
The independent compact-model argument is
\begin{gather*}
 \text{induced-metric variation}\ \Longrightarrow\
 \ker(L_N,L_T)=\KK,\\
 \text{coercivity on }\KK^\perp\ \Longrightarrow\
 \text{slice and quadratic remainder}.
\end{gather*}
Sections~\ref{sec:geometry}--\ref{sec:applications} give the first chain; Sections~\ref{sec:local} and \ref{sec:nonlinear-local} give the second. Cited classification results serve as comparisons, rather than premises of either chain. Standard results are stated with their hypotheses at the relevant points of use; their applications are accompanied by the required condition checks. Appendix~\ref{app:quadratic} retains the complete elementary ODE proofs, a quadratic majorant and an explicit single-state constraint counterexample.

\section{Geometric equations and reduction}\label{sec:geometry}
\subsection{The fixed-metric variational convention}
The pullback connection is denoted by $\nabla^X$, and
\[
 \Delta^XV=\sum_a\left(\nabla^X_{e_a}\nabla^X_{e_a}V
                  -\nabla^X_{\nabla_{e_a}e_a}V\right).
\]
For a compactly supported map variation with variational vector $V$, while $g$ is held fixed, the covariant derivative of the tension is $\JJ_XV$. Indeed, the domain frame and the domain connection are fixed, while
$\nabla_{\partial t}dX_t(e_a)=\nabla^{X_t}_{e_a}\partial_tX_t$. Commuting the two covariant derivatives therefore gives
\begin{equation}\label{eq:tension-variation-detail}
 \left.\nabla_{\partial t}\TT(X_t)\right|_{t=0}
 =\Delta^XV+\sum_a\bar R(V,dX(e_a))dX(e_a)=\JJ_XV.
\end{equation}
The last equality uses antisymmetry of the first two curvature arguments and fixes the sign in \eqref{eq:J-intro}. The operator is formally self-adjoint by integration by parts and the curvature symmetries.

More generally let $\Psi'(t)=F(t)$ for $t>0$. Differentiation of
$\int\Psi(|\TT(X)|)\dd\mu_g$ on a region where $\TT\ne0$ gives
\[
 \int\left\langle\frac{F(|\TT|)}{|\TT|}\TT,\JJ_XV\right\rangle\dd\mu_g
 =\int\ip{\JJ_XW_F}{V}\dd\mu_g.
\]
Thus $\sigma=0$ in \eqref{eq:spectral-main} has this local variational interpretation. No variational interpretation is asserted for arbitrary $\sigma$. For $F(t)=t^{p-1}$ this is the convention in \cite[eq.~(1.1)]{HL}. The description does not vary the induced metric along with the map in taking the first variation.

\subsection{Gauss--Codazzi and the normal equation}
With $h(U,V)=\ip{AU}{V}$, the defining decompositions are
\[
 \bar\nabla_UdX(V)=dX(\nabla_UV)+h(U,V)\nu,
 \qquad \bar\nabla_U\nu=-dX(AU).
\]
\begin{proposition}[Gauss--Codazzi identities with the chosen signs]\label{prop:Gauss-Codazzi}
For a hypersurface in ambient sectional curvature $c$, the preceding decompositions imply
\begin{align}
 R(U,V)W&=c\bigl(\ip VW U-\ip UW V\bigr)
          +\ip{AV}W AU-\ip{AU}W AV,\label{eq:gauss}\\
 (\nabla_UA)V&=(\nabla_VA)U.\label{eq:codazzi}
\end{align}
In particular, $\Div A=\nabla\tau$.
\end{proposition}
\begin{proof}
Identify tangent vectors with their images under $dX$. Expanding the curvature commutator by the two decompositions yields
\begin{align*}
 (\bar R(U,V)W)^\top
   &=R(U,V)W-h(V,W)AU+h(U,W)AV,\\
 \ip{\bar R(U,V)W}{\nu}
   &=(\nabla_Uh)(V,W)-(\nabla_Vh)(U,W).
\end{align*}
In a space form the left sides are respectively
$c(\ip VW U-\ip UW V)$ and zero. These give \eqref{eq:gauss} and \eqref{eq:codazzi}. Since each $\nabla_UA$ is self-adjoint, contraction in an orthonormal frame gives, for every tangent $Z$,
\[
 \ip{\Div A}{Z}
 =\sum_a\ip{(\nabla_{e_a}A)Z}{e_a}
 =\sum_a\ip{(\nabla_ZA)e_a}{e_a}=Z\tau.
\]
\end{proof}
At a point where $\nabla_{e_a}e_b=0$, differentiating the Weingarten formula gives
\[
 \nabla^X_{e_a}\nabla^X_{e_a}\nu
 =-\nabla_{e_a}(Ae_a)-h(e_a,Ae_a)\nu.
\]
Summing over $a$ and applying Proposition~\ref{prop:Gauss-Codazzi} proves
\[
 \Delta^X\nu=-\Div A-|A|^2\nu=-\nabla\tau-|A|^2\nu.
\]
Since $-\sum_a\bar R(dX(e_a),f\nu)dX(e_a)=ncf\nu$, the product rule
\[
 \Delta^X(f\nu)=(\Delta f)\nu
       +2\sum_a(e_af)\nabla^X_{e_a}\nu+f\Delta^X\nu
\]
proves
\begin{equation}\label{eq:product}
 \JJ_X(f\nu)=\bigl(\Delta f-(|A|^2-nc)f\bigr)\nu
                  -2A\nabla f-f\nabla\tau.
\end{equation}
Consequently \eqref{eq:spectral-main} is precisely
\begin{equation}\label{eq:two-equations}
 \Delta\Phi(\tau)=(|A|^2-\Lambda)\Phi(\tau),
 \qquad \bigl(2\Phi'(\tau)A+\Phi(\tau)\Id\bigr)\nabla\tau=0.
\end{equation}
The ordinary case agrees with \cite[Theorem~2.1 and Corollary~2.3]{Ou10}. This calculation also fixes the sign of the spectral parameter.

If $\tau$ is nonzero and locally constant, \eqref{eq:two-equations} implies $|A|^2=\Lambda$. In particular no such region exists when $\Lambda\le0$. If $n=1$, the tangential equation is
\[
 (2\tau\Phi'(\tau)+\Phi(\tau))\nabla\tau=0.
\]
Its coefficient is nonzero when $F,F'>0$ and $\tau\ne0$, so all nonzero-curvature regions have constant curvature. Their values are $\tau^2=\Lambda$; continuity shows that curvature is locally constant everywhere. This proves all dimension-one assertions. Henceforth $n\ge2$, and $d=n-1$.

\subsection{The response change of variable}
Orient a nonzero gradient region so that $\tau=-t<0$. Let $k_F$ be as in \eqref{eq:k-of-t} and set $\kappa=k_F(t)$. The derivative
\begin{equation}\label{eq:k-prime}
 k_F'(t)=\frac{R_F(t)-tR_F'(t)}{2R_F(t)^2}>0
\end{equation}
and the bounds $t/(2r_+)\le k_F(t)\le t/(2r_-)$ prove that $k_F$ is an increasing analytic diffeomorphism. Write its inverse as $t=T_F(k)$ and define
\begin{equation}\label{eq:Theta-Y}
 \Theta(k)=T_F(k)+k,
 \qquad Y(k)=F(T_F(k)).
\end{equation}
Both are analytic, $\Theta',Y'>0$, $\Theta$ maps $(0,\infty)$ onto itself, and
\begin{equation}\label{eq:Theta-comparison}
 (1+2r_-)k\le\Theta(k)\le(1+2r_+)k.
\end{equation}
Since $\Phi(-t)=-F(t)$ and $\Phi'(-t)=F'(t)$, the tangential equation becomes
\[
 (2F'(t)A-F(t)\Id)\nabla t=0,
 \qquad A\nabla t=k_F(t)\nabla t.
\]
Multiplying the last identity by $k_F'(t)>0$ gives
$A\nabla\kappa=\kappa\nabla\kappa$.

\subsection{Regular gradient geometry}
The local coordinate facts used next are the following standard forms of the inverse function and rank theorems; see \cite[Theorems~4.5 and~4.12]{Lee}.
\begin{standard}[Inverse function and scalar submersion principles]\label{std:inverse}
A smooth map between equal-dimensional manifolds without boundary has a smooth local inverse at a point where its differential is invertible. In particular, a smooth scalar function $q$ with $dq\ne0$ can be used as the first coordinate in a local chart. On a sufficiently small product chart $(q,z_2,\ldots,z_m)$, if a smooth function $v$ has zero derivative in every direction tangent to the level sets of $q$, then $v=\varphi(q)$ for a smooth one-variable function $\varphi$.
\end{standard}
For the last assertion, $\partial_{z_j}v=0$ for every $j\ge2$; connected coordinate fibers make $v$ independent of those coordinates. All uses below are on interior coordinate neighborhoods.

\begin{lemma}[Gradient geometry]\label{lem:gradient}
On a sufficiently small constant-multiplicity neighborhood with $d\kappa\ne0$, $\kappa$ is a simple principal curvature. For $E=\nabla\kappa/|\nabla\kappa|$ one has
\[
 \nabla_EE=0,
 \qquad Z|\nabla\kappa|=0\quad(Z\perp E).
\]
There is a local signed distance coordinate $s$ with $\nabla s=E$ and $\kappa=\kappa(s)$.
\end{lemma}
\begin{proof}
Such a neighborhood exists inside any nonempty open set. Choose a point where the number of distinct eigenvalues of $A$ is maximal there, separate the spectral clusters by disjoint intervals and shrink. No cluster can split without contradicting maximality. Here is a real-variable verification of smoothness. Let the separated eigenvalues be $u_1,\ldots,u_q$, with fixed multiplicities $m_1,\ldots,m_q$. For clarity, if $\xi_1(C)\le\cdots\le\xi_n(C)$ are the eigenvalues of a symmetric matrix $C$, its min--max formula gives
\[
 \xi_j(C)=\min_{\dim L=j}\ \max_{\substack{w\in L\\|w|=1}}\ip{Cw}{w},
 \qquad |\xi_j(C)-\xi_j(\widetilde C)|\le\|C-\widetilde C\|.
\]
Thus the spectral clusters are continuous in a local orthonormal trivialization. For $1\le j\le q$, put
\[
 s_j=\tr A^j=\sum_{\alpha=1}^q m_\alpha u_\alpha^j,
 \qquad
 \det\left(\frac{\partial s_j}{\partial u_\alpha}\right)
   =\left(\prod_{j=1}^qj\right)\left(\prod_{\alpha=1}^qm_\alpha\right)
       \prod_{\alpha<\beta}(u_\beta-u_\alpha)\ne0.
\]
The displayed Jacobian is the weighted Vandermonde matrix of the power-sum map. Standard theorem~\ref{std:inverse} supplies its smooth local inverse; composing with the smooth functions $s_j$ makes all $u_\alpha$ smooth. The polynomial projections
\[
 P_\alpha=\prod_{\beta\ne\alpha}
           \frac{A-u_\beta\Id}{u_\alpha-u_\beta}
\]
are therefore smooth. Applying them to a fixed local frame and using Gram--Schmidt gives local smooth orthonormal eigenfields.

If $AY=\kappa Y$ for a local unit field $Y\perp E$, expand
\[
 (\nabla_YA)E=(Y\kappa)E+(\kappa\Id-A)\nabla_YE.
\]
Since $A$ is self-adjoint, this gives
\[
 \begin{aligned}
 \ip{(\nabla_YA)E}{Y}
 &=(Y\kappa)\ip EY+\ip{\nabla_YE}{(\kappa\Id-A)Y}=0,\\
 \ip{(\nabla_EA)Y}{Y}&=E\kappa.
 \end{aligned}
\]
Codazzi would therefore imply $E\kappa=0$, contradicting
$E\kappa=|\nabla\kappa|>0$. Thus the gradient eigenvalue is simple;
no simplicity of a transverse eigenvalue has been used. For an eigenvector $AY=\lambda Y$ with $\lambda\ne\kappa$, the $E$ component of Codazzi gives
\[
 Y\kappa=(\kappa-\lambda)\ip{\nabla_EE}{Y}.
\]
The left side is zero, and hence $\nabla_EE=0$. Put $v=|\nabla\kappa|$. For $Z\perp E$, Hessian symmetry gives the full calculation
\[
 \begin{aligned}
 Zv&=\ip{\nabla_Z(vE)}E=\Hess\kappa(Z,E)
     =\Hess\kappa(E,Z)\\
   &=\ip{\nabla_E(vE)}Z=v\ip{\nabla_EE}Z=0.
 \end{aligned}
\]
Because $d\kappa\ne0$, Standard theorem~\ref{std:inverse} applies with $q=\kappa$. The identity $Zv=0$ on $\ker d\kappa$ then gives $v=\varphi(\kappa)>0$ on a smaller product chart. The coordinate
$s=\int_{\kappa(p)}^\kappa\varphi(u)^{-1}\dd u$ has $\nabla s=E$.
\end{proof}
Such a neighborhood can be chosen inside every nonempty open gradient region. In particular, a strict violation of an open amplitude bound persists on a nearby constant-multiplicity neighborhood; this observation will handle points of changing transverse multiplicity in Section~\ref{sec:applications}.

\subsection{The commuting transverse system}
\begin{lemma}[Commuting transverse equations]\label{lem:matrix}
Define on $E^\perp$
$S=A|_{E^\perp}$ and $BZ=-\nabla_ZE$. Along an $E$-curve, in parallel identification,
\begin{equation}\label{eq:matrix}
 S'=(S-\kappa\Id)B,
 \qquad B'=B^2+\kappa S+c\Id,
 \qquad SB=BS.
\end{equation}
A parallel orthonormal frame diagonalizes both throughout a short curve. Its diagonal entries satisfy
\begin{equation}\label{eq:pairs}
 \lambda_i'=(\lambda_i-\kappa)\mu_i,
 \qquad \mu_i'=\mu_i^2+\kappa\lambda_i+c,
 \qquad 1\le i\le d.
\end{equation}
\end{lemma}
\begin{proof}
Both $S$ and $B$ are self-adjoint; for $Z,W\perp E$,
$\ip{BZ}{W}=-v^{-1}\Hess\kappa(Z,W)$. Parallel transport preserves $E^\perp$ because $\nabla_EE=0$. Codazzi and $Z\kappa=0$ give
\[
 (\nabla_EA)Z=(\nabla_ZA)E=(S-\kappa\Id)BZ.
\]
The left side is symmetric. Taking the transpose gives $SB=BS$.
For $Z$ parallel along $E$, extend it by parallel transport along nearby $E$-curves; then $[E,Z]=BZ$. The definition of curvature now gives the full Riccati calculation
\[
 \begin{aligned}
 R(E,Z)E
 &=\nabla_E\nabla_ZE-\nabla_Z\nabla_EE-\nabla_{[E,Z]}E\\
 &=\nabla_E(-BZ)-\nabla_{BZ}E=-B'Z+B^2Z.
 \end{aligned}
\]
On the other hand, $AE=\kappa E$, $AZ=SZ$ and \eqref{eq:gauss} give
$R(E,Z)E=-\kappa SZ-cZ$. Equating the two expressions proves the second equation in \eqref{eq:matrix}.

At an initial point, commuting symmetric matrices are simultaneously diagonalizable: diagonalize $S$ and then diagonalize $B$ on each eigenspace, which it preserves. The right side of the matrix system preserves diagonal pairs. Local uniqueness for the nonautonomous matrix ODE with prescribed smooth $\kappa(s)$ therefore preserves diagonality in this fixed parallel frame. Repeated transverse eigenvalues present no obstruction. This simultaneous diagonalization is only along the chosen curve; it neither constructs nor assumes an integrable principal-coordinate net on a neighborhood of the hypersurface.
\end{proof}
\subsection{The complete necessary system}
The trace, divergence and normal equation now give
\begin{equation}\label{eq:full-ode}
 \sum_i\lambda_i=-\Theta(\kappa),\qquad
 y''-\mathsf T y'=y(\kappa^2+\mathsf Q-\Lambda),
 \qquad y=Y(\kappa),
\end{equation}
where $\mathsf T=\sum_i\mu_i$ and $\mathsf Q=\sum_i\lambda_i^2$.
Indeed $\tau=\kappa+\tr S=-T_F(\kappa)$ and $\Div E=-\mathsf T$. Specifically, $\kappa$, $\tau$ and $y$ depend only on $s$ on this neighborhood, and
\[
 \Delta y=\Div(y'E)=y''+y'\Div E=y''-\mathsf T y'.
\]
Thus $\kappa,\tau,y$ are constant along the local level leaves, whereas
no such transverse constancy is asserted for the individual $\lambda_i$ or $\mu_i$:
their equations are restricted to a single $E$-curve. In particular, simultaneous
diagonalization along that curve does not assert an integrable principal-coordinate
net on an open hypersurface neighborhood. The power case has
\[
 r=p-1,\quad \Theta(k)=(2r+1)k,\quad Y(k)=(2r)^r k^r;
\]
the positive constant multiplying $Y$ is immaterial. The pair equations alone are not sufficient: the open geometric interval also supplies the normal equation in \eqref{eq:full-ode}. This interval identity is what will survive analytic continuation. The curvature pair system agrees, in the ordinary-law case, with \cite[Proposition~4.7]{KU}, but the trace law in \eqref{eq:full-ode} is allowed to be nonlinear here.

\section{Analytic continuation of the necessary system}\label{sec:continuation}
\subsection{Local existence, uniqueness and analyticity}
The following continuation principle concerns a finite-dimensional state, not an extended immersion. Its full elementary proof is deferred to Appendix~\ref{app:odeproof} so that the constraint argument remains uninterrupted.
\begin{lemma}[Local existence, analytic persistence and continuation]\label{lem:analytic}
Let $\Omega\subset\R^m$ be open and $\mathcal V:\Omega\to\R^m$ real analytic. Every initial state has a unique maximal solution to $z'=\mathcal V(z)$. The solution is real analytic in the interior of its maximal interval. A real-analytic function defined on all of $\Omega$ that vanishes along the state on a nonempty open time subinterval vanishes along it on the entire connected maximal interval. If the state stays in a compact subset of $\Omega$ at a finite endpoint, the solution extends past that endpoint.
\end{lemma}
\begin{proof}
A complete proof, including the coefficient majorant and the open-interval identity principle, is given in Appendix~\ref{app:odeproof}. It uses only the contraction principle and finite-dimensional power series.
\end{proof}
We also use the elementary analytic inverse principle. If $q$ is real analytic and $q'(x_0)\ne0$, the solution of
\[
 x'(u)=\frac1{q'(x(u))},\qquad x(q(x_0))=x_0
\]
is analytic by Lemma~\ref{lem:analytic}, and differentiation gives $q(x(u))=u$. Thus the local inverse is analytic. This applies to $k_F$ and $\Theta$ wherever their derivatives are positive.

In the power-law specialization, the state field is polynomial and extends to the entire finite-dimensional state space. Real analyticity on an open domain, together with explicit avoidance of its boundary, is enough for the general law. Appendix~\ref{app:quadratic} also gives a real coefficient-majorant proof in the quadratic case.

\subsection{The state domain and the persistent constraint}
The response functions in \eqref{eq:Theta-Y} satisfy the following hypotheses, under which we formulate the continuation argument: $\Theta,Y:(0,\infty)\to(0,\infty)$ are real analytic, $\Theta',Y'>0$, $\Theta$ is onto $(0,\infty)$, and
\begin{equation}\label{eq:abstract-comparison}
 0<a_-k\le\Theta(k)\le a_+k\quad(k>0)
\end{equation}
for fixed $a_-,a_+>0$. Let
\[
 \Omega=\{z=(\lambda,\mu):\textstyle\sum_i\lambda_i<0\},
 \qquad k(z)=\Theta^{-1}\left(-\textstyle\sum_i\lambda_i\right).
\]
Substitute $k(z)$ for $\kappa$ in \eqref{eq:pairs}. This gives a real-analytic vector field on $\Omega$, in which the trace is built in. To make the differentiation of the constraint explicit, write $D$ for differentiation along this field and set
$\mathsf U=\sum_i(\lambda_i-k)\mu_i$. Differentiating $\Theta(k)=-\sum_i\lambda_i$ once and twice gives
\[
 Dk=-\frac{\mathsf U}{\Theta'(k)},\qquad
 D^2k=-\frac{D\mathsf U+\Theta''(k)(Dk)^2}{\Theta'(k)},
\]
where the pair equations yield
\[
 \begin{aligned}
 D\mathsf U
 &=\sum_i\{(D\lambda_i-Dk)\mu_i+(\lambda_i-k)D\mu_i\}\\
 &=2\sum_i(\lambda_i-k)\mu_i^2-\mathsf T Dk
             +\sum_i(\lambda_i-k)(k\lambda_i+c).
 \end{aligned}
\]
Every term is analytic on $\Omega$, since $\Theta'(k)>0$. The chain rule for $y=Y(k)$ therefore expresses the normal equation as the analytic residual
\begin{equation}\label{eq:general-residual}
 \RR(z)=Y''(k)(Dk)^2+Y'(k)D^2k
        -\mathsf T Y'(k)Dk
        -Y(k)(k^2+\mathsf Q-\Lambda)=0.
\end{equation}
To identify the geometric and continued states, let $z_{\mathrm{geom}}(s)$ be
the state on the original open geometric interval, and let $z_{\max}(s)$ be the
maximal solution of the analytic vector field with the same state at one interior
time. Both solve the same first-order state equation. Local uniqueness in
Lemma~\ref{lem:analytic} gives $z_{\max}=z_{\mathrm{geom}}$ on a nonempty open
time interval. The geometric normal equation consequently supplies
\[
 \RR(z_{\max}(s))=0
\]
on that entire interval. The state $z_{\max}$ is analytic, so the real-analytic
identity principle makes this equality hold throughout its connected maximal
interval. This argument neither requires the original immersion to be analytic
nor extends that immersion beyond its original domain. The zero set of $\RR$
is not asserted to be invariant through every one of its points: a single equality
$\RR(z_0)=0$ is insufficient, as Appendix~\ref{app:quadratic} demonstrates.

Define $K_\Lambda=0$ when $\Lambda\le0$. When $\Lambda>0$, define $K_\Lambda$ by
\begin{equation}\label{eq:KLambda}
 K_\Lambda^2+\frac{\Theta(K_\Lambda)^2}{d}=\Lambda.
\end{equation}
The function $G(k)=k^2+\Theta(k)^2/d-\Lambda$ is strictly increasing for $k>0$, and the comparison \eqref{eq:abstract-comparison} proves existence and uniqueness of this root. We shall exclude a trajectory with
\begin{equation}\label{eq:supercritical}
 \kappa(0)>0,\qquad \kappa'(0)>0,\qquad
 \kappa(0)\ge K_\Lambda.
\end{equation}
The initial inequality is vacuous beyond positivity if $\Lambda\le0$.

\begin{lemma}[Persistent positivity]\label{lem:flux}
On the maximal forward interval $[0,b)$ of such a trajectory, put
\[
 j_i(s)=\exp\left(-\int_0^s\mu_i(u)\dd u\right)>0,
 \qquad J=\prod_i j_i.
\]
Then
\begin{equation}\label{eq:flux-positive}
 \kappa'>0,\quad\kappa\ge\kappa(0)>0,\quad
 (y'J)'=y(\kappa^2+\mathsf Q-\Lambda)J\ge0,
 \quad y'J\ge y'(0)>0.
\end{equation}
The trajectory cannot reach the boundary of $\Omega$.
\end{lemma}
\begin{proof}
By Cauchy--Schwarz and the trace,
$\mathsf Q\ge\Theta(\kappa)^2/d$. As long as $\kappa'>0$, one has $\kappa\ge\kappa(0)$ and the potential $\kappa^2+\mathsf Q-\Lambda$ is nonnegative. The normal equation and $J'=-\mathsf TJ$ then give
\[
 (y'J)'=y(\kappa^2+\mathsf Q-\Lambda)J\ge0,
 \qquad y'J\ge y'(0)>0.
\]
If $\kappa'$ had a first zero at an interior time, these inequalities would persist there by continuity; since $J>0$ and $Y'>0$, they would give $\kappa'=y'/Y'>0$ at that time, a contradiction. Thus $\kappa'>0$ throughout the maximal forward interval and \eqref{eq:flux-positive} follows. In particular,
$\sum_i\lambda_i=-\Theta(\kappa)\le-\Theta(\kappa(0))<0$, so bounded states stay a positive distance from the boundary of $\Omega$.
\end{proof}
This is the precise point at which the sign of the normal equation is used. No convexity or other geometric condition is assumed to persist on an abstract extended trajectory.
\subsection{An integrable-coefficient estimate}
We will use the following form of Gronwall's inequality; compare \cite[Lemma~2.7]{Teschl}. Its last assertion is useful for the auxiliary moving frames.
\begin{standard}[Gronwall estimate and a terminal linear-system limit]\label{std:Gronwall}
Suppose $u\ge0$ is continuous on $[s_0,b)$, $a\ge0$ is locally integrable there, and
\[
 u(s)\le A+\int_{s_0}^s a(t)u(t)\dd t,\qquad A\ge0.
\]
Then $u(s)\le A\exp(\int_{s_0}^s a(t)\dd t)$. In particular, if $b<\infty$ and a finite-dimensional linear system $z'=M(s)z+f(s)$ has continuous coefficients on $[s_0,b)$ with
$\int_{s_0}^b(\|M\|+|f|)\dd s<\infty$, every solution has a finite coordinate limit at $b$.
\end{standard}
\begin{proof}
For $A>0$, the function $w(s)=A+\int_{s_0}^s au$ satisfies $u\le w$ and $w'\le aw$ almost everywhere. Multiplying by $\exp(-\int_{s_0}^s a)$ proves the bound. Replacing $A=0$ by $A+\varepsilon$ and letting $\varepsilon\downarrow0$ handles that case. For the linear system apply the bound to $|z(s)|$, with
$A=|z(s_0)|+\int_{s_0}^b|f|$ and $a=\|M\|$. This bounds $z$; hence $|z'|\le\|M\|\sup|z|+|f|$ is integrable. The estimate
$|z(v)-z(u)|\le\int_u^v|z'|$ makes $z$ Cauchy at $b$.
\end{proof}

\section{A common maximal-endpoint theorem}\label{sec:endpoint}
We prove the analytic statement needed for all three signs of the ambient curvature. All index sets in this section are finite, and indices are counted with multiplicity. The decisive conclusion is positivity of the complete active trace, not positivity of its individual summands.
\begin{theorem}[No supercritical increasing trajectory]\label{thm:endpoint}
Let $d\ge1$ be an integer and let $c,\Lambda\in\R$. Suppose that
$\Theta,Y:(0,\infty)\to(0,\infty)$ are real analytic,
\[
 \Theta'(k)>0,\qquad Y'(k)>0\quad(k>0),\qquad
 \Theta((0,\infty))=(0,\infty),
\]
and that, for constants $0<a_-\le a_+<\infty$,
\[
 a_-k\le\Theta(k)\le a_+k\qquad(k>0).
\]
Set $K_\Lambda=0$ if $\Lambda\le0$; if $\Lambda>0$, let $K_\Lambda$ be the unique positive solution of
\[
 K_\Lambda^2+\frac{\Theta(K_\Lambda)^2}{d}=\Lambda.
\]
There is no nonempty open interval $I$ carrying smooth real functions $\kappa>0$, $\lambda_i,\mu_i$ ($1\le i\le d$) such that, with $y=Y(\kappa)$,
\begin{equation}\label{eq:endpoint-system}
 \begin{aligned}
 \lambda_i'&=(\lambda_i-\kappa)\mu_i,
 &\quad \mu_i'&=\mu_i^2+\kappa\lambda_i+c,\qquad 1\le i\le d,\\
 \sum_{i=1}^d\lambda_i&=-\Theta(\kappa),
 &\quad y''-\Bigl(\sum_{i=1}^d\mu_i\Bigr)y'
   &=y\Bigl(\kappa^2+\sum_{i=1}^d\lambda_i^2-\Lambda\Bigr),
 \end{aligned}
\end{equation}
and, at some $s_0\in I$,
\[
 \kappa'(s_0)>0,\qquad \kappa(s_0)\ge K_\Lambda.
\]
The derivatives in \eqref{eq:endpoint-system} are with respect to the interval parameter $s$.
\end{theorem}

\subsection{Proof setting and endpoint alternatives}
Suppose, for contradiction, that the interval in Theorem~\ref{thm:endpoint} exists, and translate $s_0$ to $0$. The system \eqref{eq:endpoint-system} is exactly \eqref{eq:pairs}, \eqref{eq:full-ode}. Continue its analytic state as in Section~\ref{sec:continuation} to the maximal forward interval $[0,b)$. This hypothetical trajectory is the standing assumption in the seven lemmas below. Lemma~\ref{lem:flux} gives $\kappa'>0$ and $y'J\ge y'(0)>0$ throughout.

The argument first excludes a finite curvature limit, then rules out infinite total turning, and finally treats the finite-time, finite-turning alternative. The table records this exhaustive division into cases.

\begin{center}
\small
\renewcommand{\arraystretch}{1.16}
\begin{tabular}{@{}p{.29\textwidth}p{.47\textwidth}p{.17\textwidth}@{}}
\toprule
Possible alternative & Contradiction or reduction & Reference \\
\midrule
$b<\infty$, $\kappa$ bounded & A compact state extends; a vanishing focal factor makes $y'$ nonintegrable by the flux bound. & Lemma~\ref{lem:end-curvature} \\
$b=\infty$, $\kappa$ bounded & A Riccati lower barrier forces $y'$ to stay positive, contradicting bounded $y$. & Lemma~\ref{lem:end-curvature} \\
$\kappa\to\infty$, $\int_0^b\kappa=\infty$ & A support distance contradicts the shrinking-center distance bound. & Lemmas~\ref{lem:end-centers}--\ref{lem:end-turn} \\
$\kappa\to\infty$, $\int_0^b\kappa<\infty$ & The endpoint is finite; exact coordinates make the complete active trace positive. & Lemmas~\ref{lem:end-terminal}--\ref{lem:end-positive} \\
\bottomrule
\end{tabular}
\end{center}

\subsection{Auxiliary focal representation}
\begin{lemma}[Auxiliary focal representation]\label{lem:end-focal}
The continued state admits a unit-speed auxiliary model curve $\gamma$ and constant focal vectors $v_i$. The positive factors $j_i$ have the affine or linear representations below and satisfy \eqref{eq:scalar-focal}--\eqref{eq:focal-norm}.
\end{lemma}
\begin{proof}
For $c=0$, solve in an oriented Euclidean plane
\[
 \gamma'=e,\qquad e'=\kappa N,\qquad N'=-\kappa e,
 \qquad\gamma(0)=0
\]
with an orthonormal initial frame. For $c\ne0$, use the standard two-dimensional space-form model in $\R^3$ for $c>0$, or in Lorentz $\R^{2,1}$ for $c<0$, and solve
\begin{equation}\label{eq:spaceform-frame}
 \gamma'=e,\qquad e'=-c\gamma+\kappa N,\qquad N'=-\kappa e.
\end{equation}
Choose the initial data with $\ip\gamma\gamma=1/c$, $\ip ee=\ip NN=1$, and zero mixed products; in the hyperbolic case choose the future sheet. For example, if $\mathsf F=(\gamma,e,N)$ is the coordinate matrix and $\mathsf G$ its Gram matrix, then
\[
 \mathsf F'=\mathsf F\mathsf M,\qquad
 \mathsf M=\begin{pmatrix}0&-c&0\\1&0&-\kappa\\0&\kappa&0\end{pmatrix},
 \qquad \mathsf G'=\mathsf M^{\mathsf T}\mathsf G+\mathsf G\mathsf M.
\]
The constant matrix $\operatorname{diag}(1/c,1,1)$ solves the last equation, so linear uniqueness preserves the required Gram identities. Linear ODE existence on each compact time interval gives the frame on $[0,b)$. Thus $\gamma$ is a unit-speed curve in the model, not a curve asserted to lie in the original hypersurface.

For $c=0$, define $v_i=j_i(\mu_ie+\lambda_iN)$. For $c\ne0$, define
\begin{equation}\label{eq:focal-vector}
 v_i=j_i(-c\gamma+\mu_i e+\lambda_iN).
\end{equation}
Here is the cancellation that gives conservation. Set
$q_i=-c\gamma+\mu_i e+\lambda_iN$, omitting the zero radial term in the flat case. The frame equations give
\[
 \begin{aligned}
 q_i'&=-c\mu_i\gamma+(\mu_i'-c-\kappa\lambda_i)e
                      +(\lambda_i'+\kappa\mu_i)N\\
     &=-c\mu_i\gamma+\mu_i^2e+\mu_i\lambda_iN=\mu_iq_i.
 \end{aligned}
\]
Thus $v_i'=j_i'q_i+j_iq_i'=0$, because $j_i'=-\mu_i j_i$. Set
\[
 h_i=\ip{v_i}e,\qquad \ell_i=\ip{v_i}N.
\]
For $c\ne0$, taking the scalar product of \eqref{eq:focal-vector} with $\gamma$ gives $\ip{v_i}{\gamma}=-j_i$. For $c=0$, the derivative of $j_i+\ip{v_i}{\gamma}$ is $-\mu_i j_i+h_i=0$; its initial value is $1$. Consequently $j_i=1-\ip{v_i}\gamma$ in the flat case and $j_i=-\ip{v_i}\gamma$ otherwise. In every case,
\begin{equation}\label{eq:scalar-focal}
 \lambda_i=\frac{\ell_i}{j_i},\quad
 \mu_i=\frac{h_i}{j_i},\quad
 j_i'=-h_i,\quad h_i'=cj_i+\kappa\ell_i,\quad
 \ell_i'=-\kappa h_i,\quad j_i''=-cj_i-\kappa\ell_i.
\end{equation}
For $c\ne0$, the first three formulas also imply
\begin{equation}\label{eq:focal-norm}
 h_i^2+\ell_i^2=a_i^2-cj_i^2,
 \qquad a_i^2=\ip{v_i}{v_i}.
\end{equation}
The same identity holds for $c=0$, with $a_i^2=|v_i|^2$. We do not assume $a_i^2>0$ for an arbitrary hyperbolic factor.

\end{proof}

\subsection{Curvature divergence}
\begin{lemma}[Curvature divergence]\label{lem:end-curvature}
Along the hypothetical maximal trajectory, $\kappa(s)\to+\infty$ as $s\uparrow b$.
\end{lemma}
\begin{proof}
By strict monotonicity, $\kappa$ has a limit in $(0,\infty]$. Suppose it is finite. If $b<\infty$, the coefficients of the frame equation are bounded and the frame extends continuously to $b$. Each $j_i$, being affine or linear in $\gamma$, has a nonnegative finite limit. If all limits are positive, \eqref{eq:scalar-focal} bounds all state variables. Their bounded set stays away from $\partial\Omega$, so Lemma~\ref{lem:analytic} contradicts maximality. If instead $j_i\to0$ for some index, $h_i$ is bounded and
\[
 0<j_i(s)=\int_s^b h_i(u)\dd u\le C(b-s).
\]
The other factors are bounded above, hence $J(s)\le C_1(b-s)$. The positive flux in \eqref{eq:flux-positive} yields $y'(s)\ge C_2/(b-s)$. For a fixed terminal time $s_0<b$ this implies
\[
 y(s)-y(s_0)\ge C_2\log\frac{b-s_0}{b-s}\longrightarrow+\infty,
\]
contrary to boundedness of $y=Y(\kappa)$.

If $b=\infty$ and $\kappa\le K_0$, summing the Riccati equations gives
\begin{equation}\label{eq:T-barrier}
 \mathsf T'=\sum_i\mu_i^2-\kappa\Theta(\kappa)+dc
 \ge\frac{\mathsf T^2}{d}-C_0.
\end{equation}
Choose $C_1>0$ so that $\mathsf T(0)>-C_1$ and $C_1^2/d>C_0$. At $\mathsf T=-C_1$ the derivative is strictly positive, so a first downward crossing is impossible. Thus $\mathsf T\ge-C_1$. Choose any $s_1>0$. Since $\kappa(s_1)>\kappa(0)\ge K_\Lambda$, the increasing function $G$ satisfies $G(\kappa(s))\ge G(\kappa(s_1))>0$ for $s\ge s_1$. With $w=y'>0$ the normal equation implies
\[
 w'\ge-C_1w+c_*,\qquad c_*:=Y(\kappa(0))G(\kappa(s_1))>0.
\]
Multiplication by $e^{C_1s}$ and integration from $s_1$ to $s$ give
\[
 w(s)\ge e^{-C_1(s-s_1)}w(s_1)
       +\frac{c_*}{C_1}\bigl(1-e^{-C_1(s-s_1)}\bigr).
\]
In particular $w(s)\ge c_*/(2C_1)$ for all sufficiently large $s$. Integrating once more makes $y$ unbounded, again a contradiction. We have proved
\begin{equation}\label{eq:k-divergence}
 \kappa(s)\longrightarrow+\infty\quad(s\uparrow b).
\end{equation}

\end{proof}

\subsection{Shrinking centers and active factors}
\begin{lemma}[Shrinking centers and active factors]\label{lem:end-centers}
On a terminal interval the auxiliary curve has osculating centers $C(s)$ and radii $\rho(s)\downarrow0$ satisfying \eqref{eq:center-bound}--\eqref{eq:terminal-center}. The active set $I_*$ is nonempty, and each active focal vector has positive squared model norm.
\end{lemma}
\begin{proof}
Define the elementary curvature functions
\begin{equation}\label{eq:elementary-curvature}
 s_c(\rho)=\begin{cases}
 \sin(\sqrt c\,\rho)/\sqrt c,&c>0,\\
 \rho,&c=0,\\
 \sinh(\sqrt{-c}\,\rho)/\sqrt{-c},&c<0,
 \end{cases}
 \qquad C_c(\rho)=s_c'(\rho).
\end{equation}
They satisfy $C_c^2+cs_c^2=1$ and $C_c'=-cs_c$, so
\[
 \frac{\dd}{\dd\rho}\frac{C_c(\rho)}{s_c(\rho)}
 =-\frac{c s_c(\rho)^2+C_c(\rho)^2}{s_c(\rho)^2}
 =-\frac1{s_c(\rho)^2}.
\]
By \eqref{eq:k-divergence}, on a terminal interval there is a unique small positive $\rho$ with
\[
 \kappa=\frac{C_c(\rho)}{s_c(\rho)},\qquad
 \rho'=-\frac{\kappa'}{\kappa^2+c}<0,\qquad \rho\downarrow0.
\]
For $c<0$, \eqref{eq:k-divergence} ensures the required inequality $\kappa>\sqrt{-c}$ on this terminal interval.
Define the circle center by
\[
 C=C_c(\rho)\gamma+s_c(\rho)N\quad(c\ne0),
 \qquad C=\gamma+\rho N\quad(c=0).
\]
The radial geodesic formula gives $\dist(C,\gamma)=\rho$. Differentiation gives, in the nonflat case,
\[
 C'=\bigl(C_c(\rho)-\kappa s_c(\rho)\bigr)e
     +\rho'\{-c s_c(\rho)\gamma+C_c(\rho)N\}.
\]
The $e$ coefficient vanishes. The vector in braces has squared norm
$c s_c(\rho)^2+C_c(\rho)^2=1$. Thus
\begin{equation}\label{eq:center-derivative}
 C'=\rho'\{-cs_c(\rho)\gamma+C_c(\rho)N\},\qquad |C'|=-\rho'
\end{equation}
for $c\ne0$, and $C'=\rho'N$ for $c=0$. Consequently,
\begin{equation}\label{eq:center-bound}
 \dist(C(u),C(v))\le \rho(u)-\rho(v),\qquad u<v.
\end{equation}
Since $\rho\downarrow0$, \eqref{eq:center-bound} makes $C(s)$ a Cauchy curve at the terminal end, even before $b$ is known to be finite. The complete auxiliary model therefore contains a limit point $\gamma_*$; letting $v\uparrow b$ in that estimate gives the last inequality below. Also
$\dist(\gamma(s),\gamma_*)\le\rho(s)+\dist(C(s),\gamma_*)\le2\rho(s)$, so
\begin{equation}\label{eq:terminal-center}
 C(s)\to\gamma_*,\qquad\gamma(s)\to\gamma_*,\qquad
 \dist(C(s),\gamma_*)\le \rho(s).
\end{equation}
The model completeness used here is elementary: the sphere is compact, the plane is complete, and a closed hyperbolic ball is a closed bounded subset of the future hyperboloid in ambient coordinates. It is not completeness of $M$.

Each factor has a nonnegative limit. The active set
\[
 I_*:=\{i:j_i(s)\to0\}
\]
is nonempty. Indeed, if all limits were positive, \eqref{eq:focal-norm} would bound $h_i,\ell_i$, and hence all $\lambda_i$, contradicting $\sum\lambda_i=-\Theta(\kappa)\to-\infty$. An active flat vector is nonzero because $\ip{v_i}{\gamma_*}=1$. An active nonflat vector is nonzero by its initial formula \eqref{eq:focal-vector}, and $\ip{v_i}{\gamma_*}=0$; it is thus a nonzero tangent vector to the model at $\gamma_*$. In particular, for every active index,
\begin{equation}\label{eq:active-spacelike}
 a_i^2>0.
\end{equation}

\end{proof}

\subsection{Finite total curvature}
\begin{lemma}[Finite total curvature]\label{lem:end-turn}
The hypothetical trajectory has $\int_0^b\kappa(s)\dd s<\infty$ and $b<\infty$.
\end{lemma}
\begin{proof}
Fix an active vector $v$, omit its index, and set $a=\sqrt{\ip vv}>0$, as permitted by \eqref{eq:active-spacelike}. On a terminal interval, $a^2-cj^2>0$ and we can choose a continuous angle
\[
 h=\sqrt{a^2-cj^2}\cos\theta,\qquad
 \ell=\sqrt{a^2-cj^2}\sin\theta.
\]
The angle derivative can be computed without differentiating the square root:
\[
 \theta'=\frac{h\ell'-\ell h'}{h^2+\ell^2}
 =\frac{-\kappa h^2-\ell(cj+\kappa\ell)}{a^2-cj^2}.
\]
Thus \eqref{eq:scalar-focal} gives the exact formula
\begin{equation}\label{eq:phase-law}
 \theta'=-\kappa-\frac{cj\ell}{a^2-cj^2}.
\end{equation}
Its correction term tends to zero. If $\int_0^b\kappa=\infty$, then on a tail $\theta'\le-\kappa/2$ and $\theta\to-\infty$. At arbitrarily late times we would have $h=0$, $\ell=-\sqrt{a^2-cj^2}$.

Consider the signed support distance
\begin{equation}\label{eq:support-distance}
 D(x)=\begin{cases}
 \displaystyle\frac1{\sqrt c}\arcsin\!\left(\frac{\sqrt c(-\ip vx)}{a}\right),&c>0,\\[4pt]
 \displaystyle\frac{1-\ip vx}{a},&c=0,\\[4pt]
 \displaystyle\frac1{\sqrt{-c}}\operatorname{arsinh}\!\left(\frac{\sqrt{-c}(-\ip vx)}{a}\right),&c<0.
 \end{cases}
\end{equation}
It satisfies $D(\gamma_*)=0$ and $D(\gamma)>0$. We now justify its global $1$-Lipschitz property, including the spherical poles. Put $f=-\ip vx/a$ for $c\ne0$ and $f=(1-\ip vx)/a$ for $c=0$. For $c>0$, set
\[
 x_+=-\frac{v}{a\sqrt c}\in\Sph^2(c^{-1/2}).
\]
The spherical distance formula gives
\[
 \cos\!\bigl(\sqrt c\,\dist(x,x_+)\bigr)=\sqrt c\,f(x),
 \qquad
 D(x)=\frac{\pi}{2\sqrt c}-\dist(x,x_+).
\]
The second identity holds on the entire sphere, because
$\arcsin(\cos\theta)=\pi/2-\theta$ for $0\le\theta\le\pi$.
The triangle inequality therefore gives $|D(x)-D(y)|\le\dist(x,y)$ even at the poles. For $c<0$, orthogonal projection to the tangent space of the hyperboloid gives
\[
 \nabla f=-v/a-cf x,\qquad |\nabla f|^2=1-cf^2,
 \qquad \nabla D=\frac{\nabla f}{\sqrt{1-cf^2}}.
\]
Thus $|\nabla D|=1$ everywhere. Integrating along a minimizing model geodesic gives the same Lipschitz bound. For $c=0$, $D$ is affine with unit gradient, so the bound is immediate.

At a special time with $h=0$ and negative $\ell$, put $d_0=D(\gamma)$. Then $j/a=s_c(d_0)$ and $-\ell/a=C_c(d_0)$. The addition formula implies
\[
 f(C)=C_c(\rho)s_c(d_0)+s_c(\rho)C_c(d_0)=s_c(d_0+\rho).
\]
For $c>0$ choose the special time so late that $d_0+\rho<\pi/(2\sqrt c)$; for $c\le0$ there is no inverse-branch restriction. Thus $D(C)=d_0+\rho>\rho$. On the other hand, \eqref{eq:terminal-center} and the Lipschitz property give $D(C)\le \rho$. This contradiction proves
\begin{equation}\label{eq:finite-turn}
 \int_0^b\kappa(s)\dd s<\infty,\qquad b<\infty.
\end{equation}
The second assertion follows from $\kappa\ge\kappa(0)>0$.

\end{proof}

\subsection{Exact terminal coordinates}
The notation for the remaining argument is summarized below. The preceding lemma and the following proofs establish the listed limits as $s\uparrow b$.
\begin{center}
\small
\renewcommand{\arraystretch}{1.14}
\begin{tabular}{@{}p{.15\textwidth}p{.46\textwidth}p{.32\textwidth}@{}}
\toprule
Quantity & Definition or role & Terminal information \\
\midrule
$\zeta$ & $b-s$, remaining time & $\zeta\downarrow0$, $\kappa\zeta\to0$ \\
$\delta$ & $\int_s^b\kappa(u)\dd u$, remaining turn & $\delta\downarrow0$ \\
$U,V$ & Signed terminal position coordinates; see \eqref{eq:UPVB} & $U/\zeta\to1$, $0<V\le C\zeta\delta$ \\
$P,\mathsf B$ & Terminal normal components; see \eqref{eq:UPVB} & $P>0$, $P\to0$, $\mathsf B\to1$ \\
$\alpha_i,\beta_i$ & $v_i=\alpha_i e_*+\beta_iN_*$, $i\in I_*$ & $\alpha_i\ge0$, then $\alpha_i>0$ \\
$\eta_i,w$ & $\eta_i=\beta_i/\alpha_i$, $w=V/U$ & $\sum_{I_*}\eta_i=0$, $w\to0^+$ \\
\bottomrule
\end{tabular}
\end{center}
\begin{lemma}[Terminal frame and exact coordinates]\label{lem:end-terminal}
The terminal frame exists. With the coordinates in \eqref{eq:UPVB}, the estimates \eqref{eq:kt-small} and \eqref{eq:V-estimates} hold. Every active factor has the representation \eqref{eq:terminal-exact} with $\alpha_i\ge0$.
\end{lemma}
\begin{proof}
Put $\zeta=b-s$ and $\delta(s)=\int_s^b\kappa(u)\dd u$.
Monotonicity yields
\begin{equation}\label{eq:kt-small}
 0<\kappa(s)\zeta\le\delta(s)\longrightarrow0.
\end{equation}
This estimate follows from integrability and monotonicity alone; it prescribes no asymptotic power or order of blow-up.
Regard all ambient coordinates of the frame as one finite-dimensional linear system, with an ordinary Euclidean norm on its coordinate entries even when the model scalar product is Lorentzian. Its coefficient norm is at most $C(1+\kappa)$. Since $b<\infty$ and \eqref{eq:finite-turn} holds,
\[
 \int_{s_0}^b C(1+\kappa)\dd s<\infty.
\]
Standard theorem~\ref{std:Gronwall} therefore gives a finite limit of the coordinate matrix. The Gram identities persist under this limit and give an orthonormal tangent pair $e_*,N_*$ at $\gamma_*$.
Define
\begin{equation}\label{eq:UPVB}
 U=\ip{\gamma_*-\gamma}{e_*},\quad
 V=\ip{\gamma-\gamma_*}{N_*},\quad
 P=\ip N{e_*},\quad \mathsf B=\ip N{N_*}.
\end{equation}
The same expressions apply in the flat plane. Since $e(s)\to e_*$ and $N(s)\to N_*$,
\[
 U(s)=\int_s^b\ip{e(u)}{e_*}\dd u,\qquad
 P(s)=\int_s^b\kappa(u)\ip{e(u)}{e_*}\dd u.
\]
The integrands contain a scalar product converging to $1$. Hence
$U/\zeta\to1$, $\mathsf B\to1$, and, on a sufficiently short terminal interval,
\[
 \tfrac12\zeta\le U\le2\zeta,\qquad
 \tfrac12\delta\le P\le2\delta.
\]
In particular $U,P>0$ and $P\to0$. Moreover, $V''+cV=\kappa \mathsf B$, $V(b)=V'(b)=0$. The exact terminal integral is
\begin{equation}\label{eq:V-kernel}
 V(s)=\int_s^b s_c(u-s)\kappa(u)\mathsf B(u)\dd u.
\end{equation}
To justify differentiation at this possibly singular endpoint, put $q=\kappa\mathsf B$. By \eqref{eq:finite-turn}, $q\in L^1(s_0,b)$ on a terminal interval. If $I(s)$ denotes the integral on the right, boundedness of the kernels on this interval and dominated convergence give
\[
 I'(s)=-\int_s^b C_c(u-s)q(u)\dd u,\qquad
 I''(s)+cI(s)=q(s),\qquad I(b)=I'(b)=0.
\]
The second derivative is taken at interior points, where $q$ is smooth. The difference $V-I$ solves the homogeneous scalar equation with zero terminal data, and hence vanishes. On a sufficiently short terminal interval its kernel and $\mathsf B$ are positive and $s_c(u-s)\le C(u-s)$. In particular,
\[
 V(s)\le C\int_s^b(u-s)\kappa(u)\dd u
       \le C\zeta\delta,\qquad
 \kappa(s)V(s)\le C(\kappa(s)\zeta)\delta\le C\delta^2.
\]
This proves
\begin{equation}\label{eq:V-estimates}
 0<V\le C\zeta\delta,\qquad V/\zeta\to0,\qquad
 0<\kappa V\le C\delta^2\to0.
\end{equation}
Formula \eqref{eq:V-kernel} uses the appropriate space-form kernel; for $c<0$ this is the hyperbolic sine kernel.

For an active index write $v_i=\alpha_i e_*+\beta_iN_*$. In the nonflat cases this uses precisely $v_i\perp\gamma_*$. The affine or linear representations then give
\begin{equation}\label{eq:terminal-exact}
 j_i=\alpha_i U-\beta_iV>0,\qquad
 \lambda_i=\frac{\alpha_iP+\beta_i\mathsf B}{\alpha_iU-\beta_iV}.
\end{equation}
Dividing the first equation by $\zeta$ and using $U/\zeta\to1$, $V/\zeta\to0$ gives $j_i/\zeta\to\alpha_i$. Since $j_i>0$, necessarily $\alpha_i\ge0$. No division by $\alpha_i$ has been made.
\end{proof}

\subsection{Excluding tangential active factors}
\begin{lemma}[Strictly transverse incidence]\label{lem:end-transverse}
Every active factor has $\alpha_i>0$ in \eqref{eq:terminal-exact}.
\end{lemma}
\begin{proof}
Let $I_0=\{i\in I_*:\alpha_i=0\}$. We use the already constructed coordinate $V$, whose positivity and smallness are given by \eqref{eq:V-estimates}.
For $i\in I_0$, the nonzero active vector has $\beta_i\ne0$ and the exact coordinate formulas reduce to
\[
 j_i=-\beta_iV>0,\qquad
 \lambda_i=\frac{\beta_i\mathsf B}{-\beta_iV}
           =-\frac{\mathsf B}{V}.
\]
Thus $\beta_i<0$, and, without any assumption on the zero order of $j_i$,
\begin{equation}\label{eq:tangential-residue}
 V\lambda_i=-\mathsf B\longrightarrow-1\qquad(i\in I_0).
\end{equation}
If $i\in I_*\setminus I_0$, then $\alpha_i>0$, and \eqref{eq:terminal-exact} gives
\[
 \zeta\lambda_i\longrightarrow\beta_i/\alpha_i.
\]
Because $V/\zeta\to0$, it follows that $V\lambda_i\to0$ for these indices. For an inactive index, $j_i$ has a positive limit and $\ell_i$ has a finite limit, so $\lambda_i=\ell_i/j_i$ is bounded and again $V\lambda_i\to0$.
Now multiply the full trace equation by $V$:
\[
 \sum_{i=1}^d V\lambda_i=-V\Theta(\kappa).
\]
The right side tends to zero, since $0<V\Theta(\kappa)\le a_+\kappa V\to0$. The left side tends to $-|I_0|$ by the preceding limits. The sum is finite and includes all repeated factors, so these limits may be taken term by term. Hence $|I_0|=0$, which proves $\alpha_i>0$ for every active index.
\end{proof}

\subsection{Positive complete active trace}
\begin{lemma}[Positive complete active trace]\label{lem:end-positive}
The terminal representations force \eqref{eq:positive-trace} and the quantitative
bound \eqref{eq:active-trace-growth}. Consequently the trace law would bound
$\kappa$ above, contrary to \eqref{eq:k-divergence}.
\end{lemma}
\begin{proof}
Set $\eta_i=\beta_i/\alpha_i$ for $i\in I_*$ and $w=V/U>0$ near $b$. For every active index, \eqref{eq:terminal-exact} gives $\zeta\lambda_i\to\eta_i$; the inactive terms are bounded. Meanwhile
\[
 \left|\zeta\sum_i\lambda_i\right|
 =\zeta\Theta(\kappa)\le a_+\zeta\kappa\longrightarrow0.
\]
Taking the limit of the finite sum therefore gives
\begin{equation}\label{eq:eta-sum}
 \sum_{i\in I_*}\eta_i=0.
\end{equation}
Since $j_i=\alpha_iU(1-\eta_iw)>0$, all denominators below are positive. The elementary identity
\[
 \frac{\eta}{1-\eta w}=\eta+\frac{\eta^2w}{1-\eta w}
\]
and \eqref{eq:eta-sum} imply
\[
 \sum_{i\in I_*}\frac{\eta_i}{1-\eta_iw}
 =w\sum_{i\in I_*}\frac{\eta_i^2}{1-\eta_iw}.
\]
Indeed, after dividing numerator and denominator in \eqref{eq:terminal-exact} by $\alpha_iU$,
\[
 \lambda_i=\frac1U\frac{P+\eta_i\mathsf B}{1-\eta_iw}.
\]
Summing these expressions gives the exact equality
\begin{equation}\label{eq:positive-trace}
 \sum_{i\in I_*}\lambda_i
 =\frac{P}{U}\sum_{i\in I_*}\frac1{1-\eta_iw}
  +\frac{\mathsf B w}{U}\sum_{i\in I_*}\frac{\eta_i^2}{1-\eta_iw}>0.
\end{equation}
The first term is strictly positive because $I_*\ne\varnothing$, $P>0$, $U>0$,
and every denominator $1-\eta_iw$ is positive. The second term is nonnegative
because $\mathsf B>0$ and $w>0$. These are assertions about the complete active
sum; no sign is imposed on each individual $\lambda_i$.

For a quantitative contradiction, put $m=|I_*|\ge1$. The identity
\[
 \frac{P+\eta_i\mathsf B}{1-\eta_iw}
 =P+(Pw+\mathsf B)\frac{\eta_i}{1-\eta_iw}
\]
rewrites the same complete sum as
\[
 \sum_{i\in I_*}\lambda_i
 =\frac{mP}{U}
  +\frac{w(Pw+\mathsf B)}{U}
    \sum_{i\in I_*}\frac{\eta_i^2}{1-\eta_iw}
 \ge\frac{mP}{U}.
\]
On a sufficiently short terminal interval, the bounds in
Lemma~\ref{lem:end-terminal} give $P\ge\delta/2$, $U\le2\zeta$ and
$\delta\ge\kappa\zeta$. Thus
\begin{equation}\label{eq:active-trace-growth}
 \sum_{i\in I_*}\lambda_i\ge\frac m4\kappa.
\end{equation}
The inactive sum is bounded, say its absolute value is at most $C$ (take $C=0$
when no inactive index exists). The trace and \eqref{eq:abstract-comparison}
would now give
\[
 -\Theta(\kappa)=\sum_i\lambda_i\ge\frac m4\kappa-C,
 \qquad
 \left(a_-+\frac m4\right)\kappa\le C.
\]
This contradicts $\kappa\to\infty$. No asymptotic expansion or assumed blow-up
order has entered the argument.
\end{proof}

\begin{proof}[Proof of Theorem~\ref{thm:endpoint}]
Any interval satisfying the hypotheses gives the hypothetical maximal state used above. Lemma~\ref{lem:end-curvature} excludes both bounded-curvature endpoint alternatives. Lemmas~\ref{lem:end-centers} and \ref{lem:end-turn} then give a finite endpoint with finite total curvature. Lemmas~\ref{lem:end-terminal}, \ref{lem:end-transverse} and \ref{lem:end-positive} contradict the trace on the only remaining alternative. Thus no proposed interval exists.
\end{proof}

\begin{remark}
The terminal argument uses exact focal coordinates and the estimates \eqref{eq:kt-small} and \eqref{eq:V-estimates}. It applies to repeated focal factors and to arbitrary positive real power exponents without a prescribed asymptotic expansion. The final trace identity is common to all three geometries, with the space-form dependence contained in the kernel that produces $V$. Multiplication by the terminal coordinate $V$ excludes tangential contact through the exact residue $V\lambda_i=-\mathsf B$, without an auxiliary integral or a prescribed zero order. The available positive flux is $y^{\prime}J=Y^{\prime}(\kappa)\kappa^{\prime}J$, and no upper growth bound for $Y^{\prime}$ at infinity is required. In the curved models, $\int\kappa$ is total geodesic curvature, not an angle change in a fixed ambient frame; the correction in \eqref{eq:phase-law} is retained throughout.
\end{remark}

\section{Consequences in Euclidean, spherical and hyperbolic space}\label{sec:applications}
\subsection{Returning from the ODE to the hypersurface}
\begin{proof}[Proof of Theorem~\ref{thm:nonlinear}]
Suppose first $n\ge2$. On any nonzero-mean-curvature neighborhood with $d\tau\ne0$, choose the normal with $\tau<0$ and restrict to a regular spectral neighborhood as in Section~\ref{sec:geometry}. Choose the direction of the gradient curve so that $\kappa'>0$. The geometric equations supply an open interval solution of \eqref{eq:pairs}, \eqref{eq:full-ode}.

If $\Lambda\le0$, every positive initial value is supercritical in the sense of \eqref{eq:supercritical}. Theorem~\ref{thm:endpoint} rules out such an interval. A nonzero locally constant $\tau$ would give $|A|^2=\Lambda\le0$, which is also impossible. If the nonzero-mean-curvature set were nonempty, either its mean curvature would be locally constant everywhere on a connected coordinate component, or there would be a point with $d\tau\ne0$. Both alternatives are excluded. Thus the immersion is minimal.

\emph{Strictness at arbitrary gradient points.}
If $\Lambda>0$, Theorem~\ref{thm:endpoint} excludes $\kappa\ge K_\Lambda$
at a regular spectral point with $d\kappa\ne0$. Now let $q$ be any point of a
nonzero-amplitude neighborhood with $d\kappa(q)\ne0$, without a multiplicity
assumption. If $\kappa(q)>K_\Lambda$, both strict exceedance and nonzero
derivative persist on an open neighborhood. Lemma~\ref{lem:gradient} supplies
a regular spectral neighborhood inside it, a contradiction. If instead
$\kappa(q)=K_\Lambda$, follow the local gradient flow $\eta$ with
$\eta(0)=q$ and $\eta'=\nabla\kappa$. On a sufficiently short interval,
\[
 \frac{\dd}{\dd t}\kappa(\eta(t))=|\nabla\kappa(\eta(t))|^2>0.
\]
Positive time therefore produces the already excluded strict exceedance.
This use of the nonzero derivative, not continuity alone, proves strictness.
Since $k_F$ is strictly increasing, the resulting inequality is exactly
$|\tau|<t_F(\Lambda)$ at every nonzero-amplitude gradient point. At a
zero-mean-curvature point the asserted numerical inequality is automatic.

\emph{From one strict exceedance to a CMC component.}
Work on a connected component of $M$ and put $h_*=t_F(\Lambda)/n$.
The scalar $|H|$ is globally defined and is smooth on
$\mathcal U=\{|H|>h_*\}$. The gradient-point bound implies $d|H|=0$ on
$\mathcal U$, so every connected component $\mathcal U_0$ of $\mathcal U$
has one constant value $h_0>h_*$. Such a component is open because manifolds
are locally connected. If $q\in\overline{\mathcal U_0}$, continuity gives
$|H(q)|=h_0>h_*$. Choose a connected coordinate neighborhood
$B\subset\mathcal U$ of $q$. Since $B\cap\mathcal U_0\ne\varnothing$,
connectedness places $B$ in the same component $\mathcal U_0$.
Thus $q\in\mathcal U_0$, and $\mathcal U_0$ is also closed in $M$.
Connectedness forces $\mathcal U_0=M$. This proves the one-point CMC
assertion and its contrapositive, the global non-CMC bound; it does not require
the gradient set of a nonconstant smooth function to be dense.
On a nonzero CMC component the normal equation gives $|A|^2=\Lambda$, and
$\tau^2\le n|A|^2$ gives \eqref{eq:CMC-norm}. The argument is local in the
normal orientation, while $|H|$ is orientation-independent, so nonorientable
hypersurfaces are included.

For $n=1$, the tangential equation on a nonzero curvature interval is
\[
 (2\tau\Phi'(\tau)+\Phi(\tau))\tau'=0.
\]
The coefficient has the strict sign of $\tau$, since $F,F'>0$, hence $\tau'=0$. The normal equation then gives $\tau^2=\Lambda$ on that interval. Continuity precludes a boundary between a nonzero constant curvature interval and zero curvature. Thus the curvature is locally constant everywhere, and when $\Lambda\le0$ it vanishes.
\end{proof}

\subsection{Powers and the four-parameter formula}
For $r=p-1>0$, $F(t)=t^r$ has constant elasticity $r$. The functions are
\[
 k_F(t)=\frac{t}{2r},\quad \Theta(k)=(2r+1)k,
 \quad Y(k)=(2r)^rk^r.
\]
The positive constant factor $(2r)^r$ in $Y$ cancels from the homogeneous normal equation. Equation \eqref{eq:KLambda} gives
\[
 K_\Lambda=\sqrt{\frac{(n-1)\Lambda}{n-1+(2r+1)^2}},
 \qquad t_F(\Lambda)=2rK_\Lambda,
\]
which proves Theorem~\ref{thm:power}. In particular:
\begin{corollary}\label{cor:spaceforms}
In every finite dimension, minimality holds for solutions of \eqref{eq:power-main} with $p>1$ and $\sigma\ge nc$. It includes all Euclidean and hyperbolic $p$-biharmonic hypersurfaces ($\sigma=0$, $c\le0$), and, in hyperbolic space, also the spectral range $nc\le\sigma<0$. For $c>0$, $\sigma=0$, the non-CMC threshold is
\[
 |H|\le\frac{2(p-1)}{\sqrt n}
       \sqrt{\frac{(n-1)c}{n-1+(2p-1)^2}}.
\]
All assertions concern the induced metric and the explicit operator \eqref{eq:J-intro}.
\end{corollary}
For powers one can alternatively eliminate the fractional power altogether. With $k=-(2r+1)^{-1}\sum\lambda_i$, multiplication of the normal equation by $k^{2-r}$ on $k>0$ gives the polynomial residual
\begin{equation}\label{eq:power-residual}
 r kD^2k+r(r-1)(Dk)^2-r\mathsf T kDk
       -k^2(k^2+\mathsf Q-\Lambda)=0.
\end{equation}
Its coefficients may be real; the exponent $r$ need not be integral. The proof still uses persistence on an open interval, not the vanishing of this polynomial at a single initial state.

\subsection{Mixtures: the variance criterion}
\begin{proof}[Proof of Corollary~\ref{cor:mixtures}]
Let $\omega_j(t)=a_jt^{r_j}/F(t)$. These positive weights sum to one. Logarithmic differentiation gives
\[
 t\omega_j'(t)=\omega_j(t)(r_j-R_F(t)),\qquad
 tR_F'(t)=\sum_jr_j\omega_j(r_j-R_F)
         =\sum_j\omega_j r_j^2-R_F^2.
\]
Consequently
\begin{equation}\label{eq:variance}
 R_F=\sum_j\omega_jr_j,
 \qquad tR_F'=\sum_j\omega_j(r_j-R_F)^2.
\end{equation}
For a random variable taking values in $[r_-,r_+]$, the pointwise inequality $(r-r_-)(r_+-r)\ge0$ implies
\[
 \operatorname{Var}(r)\le (R_F-r_-)(r_+-R_F)
 \le\frac{(r_+-r_-)^2}{4}.
\]
Thus $R_F-tR_F'\ge r_--(r_+-r_-)^2/4>0$. The other admissibility conditions follow immediately; real powers are real analytic on the positive half-line.
\end{proof}
For example $F(t)=t+t^2$ has
\[
 R_F(t)=\frac{1+2t}{1+t},\qquad
 k_F(t)=\frac{t(1+t)}{2(1+2t)},\qquad
 k_F'(t)=\frac{1+2t+2t^2}{2(1+2t)^2}>0.
\]
For a fixed immersion, its spectral residual is the sum of the residuals for $F_1(t)=t$ and $F_2(t)=t^2$, because $\JJ_X-\sigma\Id$ is linear in the normal field. Vanishing of that sum does not imply separate vanishing of the two residuals, and the immersion equations do not admit a superposition principle. The conclusion follows instead from the nonlinear trace law and Theorem~\ref{thm:endpoint}.

\subsection{Explicit hyperbolic consequences for every response class}\label{subsec:hyperbolic-results}
The following is a specialization of the common theorem, not a separate ordinary-biharmonic argument.
\begin{corollary}[Hyperbolic nonlinear spectral rigidity]\label{cor:hyperbolic-all}
Let $a>0$ and let $X:M^n\to\Hyp^{n+1}(-a^2)$ satisfy \eqref{eq:spectral-main} for an admissible law $F$.
If $\sigma\ge-na^2$, then $X$ is minimal. If $\sigma<-na^2$ and $n\ge2$, let $t_{F,a,\sigma}>0$ be the unique root of
\begin{equation}\label{eq:hyperbolic-implicit}
 k_F(t)^2+\frac{(t+k_F(t))^2}{n-1}=-na^2-\sigma.
\end{equation}
At every gradient point, $|H|<t_{F,a,\sigma}/n$. A connected non-CMC solution satisfies $|H|\le t_{F,a,\sigma}/n$ everywhere; strict exceedance at one point forces CMC. On a nonminimal CMC component,
\begin{equation}\label{eq:hyperbolic-CMC}
 |A|^2=-na^2-\sigma,\qquad |H|\le\sqrt{-a^2-\sigma/n}.
\end{equation}
For $n=1$ the curvature is locally constant, with $\tau^2=-a^2-\sigma$ on each nonminimal component.
\end{corollary}
\begin{proof}
Substitute $c=-a^2$ in Theorem~\ref{thm:nonlinear}. Then $\Lambda=-na^2-\sigma$, and \eqref{eq:implicit-threshold} and \eqref{eq:CMC-norm} become the displayed formulas. The $n=1$ assertion follows from the one-dimensional part of that theorem.
\end{proof}
For powers $F(t)=t^{p-1}$, $p>1$, the hyperbolic gradient-point bound is explicitly
\begin{equation}\label{eq:hyperbolic-power-bound}
 |H|<\frac{2(p-1)}{n}
 \sqrt{\frac{(n-1)(-na^2-\sigma)}{n-1+(2p-1)^2}},
 \qquad n\ge2,\quad \sigma<-na^2.
\end{equation}
At $\sigma=0$, every such power-law solution is minimal. The same minimality conclusion holds for every positive mixture in Corollary~\ref{cor:mixtures}, not only for individual powers. For the nonhomogeneous example $F(t)=t+t^2$, the threshold in \eqref{eq:hyperbolic-implicit} uses
\[
 k_F(t)=\frac{t(1+t)}{2(1+2t)};
\]
no separate vanishing of the two pure-power residuals is assumed.

For comparison, the Euclidean specialization replaces the right-hand side of \eqref{eq:hyperbolic-implicit} by $-\sigma$ and has minimality for $\sigma\ge0$. The spherical specialization has right-hand side $nc-\sigma$ and minimality for $\sigma\ge nc$. The zero-parameter spherical equation is subcritical, so it gives curvature thresholds and CMC criteria rather than unconditional minimality. These distinctions explain why the same theorem has different geometric consequences in the three spaces.

\subsection{Rescaling an ambient curvature and a nonhomogeneous law}\label{subsec:scaling}
Keeping $c$ arbitrary avoids an implicit change of constitutive law. The precise scaling rule is as follows.
\begin{proposition}[Homothety of the spectral equation]\label{prop:scaling}
Let $b_0>0$ and replace the ambient and induced metrics by $b_0^2\bar g$ and $b_0^2g$. Then
\[
 \widetilde c=c/b_0^2,\qquad \widetilde\tau=\tau/b_0,
 \qquad \widetilde\sigma=\sigma/b_0^2,
 \qquad \widetilde F(t)=F(b_0t)
\]
transform \eqref{eq:spectral-main} into the same equation in the rescaled geometry. Admissibility is preserved.
\end{proposition}
\begin{proof}
Constant rescaling preserves the connections and the curvature tensor as a $(1,3)$-tensor. The unit normal becomes $\widetilde\nu=\nu/b_0$, and tracing with the rescaled metric gives $\widetilde A=A/b_0$, $\widetilde\tau=\tau/b_0$. The orthonormal tangent frame is divided by $b_0$, so $\widetilde\JJ=b_0^{-2}\JJ$ on vector fields along $X$. The chosen response satisfies
\[
 \widetilde W_{\widetilde F}=b_0^{-1}W_F,\qquad
 \widetilde\JJ\widetilde W_{\widetilde F}
 =b_0^{-3}\JJ W_F
 =\frac{\sigma}{b_0^2}\widetilde W_{\widetilde F}.
\]
Moreover, $R_{\widetilde F}(t)=R_F(b_0t)$ and
\[
 R_{\widetilde F}(t)-tR_{\widetilde F}'(t)
 =R_F(b_0t)-(b_0t)R_F'(b_0t)>0.
\]
The same elasticity bounds hold, and analyticity is preserved.
\end{proof}
In particular, normalizing $c=-a^2$ to $-1$ uses $b_0=a$, hence $\widetilde\sigma=\sigma/a^2$ and $\widetilde F(t)=F(at)$. For a pure power this changes only a positive overall factor, which cancels from the equation. For a mixture $\sum_j a_jt^{r_j}$ it changes the coefficients to $a_j a^{r_j}$ and generally changes their ratios. Thus an unchanged nonlinear law must not be silently carried through a curvature normalization.

\subsection{Umbilical examples and the exact spectral boundary}\label{subsec:umbilical-range}
A nonzero CMC umbilical hypersurface $A=\lambda\Id$ has $|\tau|=n|\lambda|$ and $|A|^2=n\lambda^2$. The two equations \eqref{eq:two-equations} reduce to
\begin{equation}\label{eq:umbilical-spectral}
 \JJ_XW_F=n(c-\lambda^2)W_F.
\end{equation}
This uses only a positive value of $F$ at the constant amplitude; no global admissibility condition is needed for the examples.

For a geodesic sphere of radius $R$, polar coordinates give the metric
\[
 dR^2+s_c(R)^2g_{\Sph^n(1)},
\]
where $s_c,C_c$ are defined in \eqref{eq:elementary-curvature}. Choose the inward radial normal. Differentiating the warping factor gives
\begin{equation}\label{eq:geodesic-radius-parameters}
 \lambda=\frac{C_c(R)}{s_c(R)},\qquad
 K_0=\lambda^2+c=\frac1{s_c(R)^2},\qquad
 \sigma=n(c-\lambda^2).
\end{equation}
For completeness, the nonflat model parametrization is
$X_R(\omega)=C_c(R)o+s_c(R)\omega$, with
$\langle o,o\rangle=1/c$, $\omega\in T_o\mathbb M_c^{n+1}$ and $|\omega|=1$.
Its angular differential is $s_c(R)d\omega$, which proves the displayed induced metric; differentiating the inward radial normal yields $A=(C_c/s_c)\Id$. The flat parametrization is $o+R\omega$.

In Euclidean space, $\lambda=1/R$, so compact round spheres realize every $\sigma<0$. In the sphere of curvature $c>0$, take $0<R<\pi/(2\sqrt c)$; then $\lambda=\sqrt c\cot(\sqrt cR)$ runs over all positive values. These compact geodesic spheres realize every $\sigma<nc$. A larger nonminimal spherical geodesic sphere is congruent to a smaller one about the antipodal center, after a normal reversal.

In $\Hyp^{n+1}(-a^2)$ the complete umbilical examples divide into three ranges:
\begin{center}
\small
\renewcommand{\arraystretch}{1.2}
\begin{tabular}{@{}lll@{}}
\toprule
Model & Positive principal curvature & Spectral range \\
\midrule
Geodesic sphere, $R>0$ & $a\coth(aR)$ & $\sigma<-2na^2$ \\
Horosphere & $a$ & $\sigma=-2na^2$ \\
Equidistant hypersurface, $R>0$ & $a\tanh(aR)$ & $-2na^2<\sigma<-na^2$ \\
\bottomrule
\end{tabular}
\end{center}
Here the last two types are noncompact. The geodesic-sphere formula follows from \eqref{eq:geodesic-radius-parameters}. To check the equidistant formula without a classification theorem, let $P\subset\Hyp^{n+1}(-a^2)$ be a totally geodesic hyperplane and $\eta$ its constant unit spacelike normal in the hyperboloid model. For $x\in P$ set
\[
 X_R(x)=\cosh(aR)x+\frac{\sinh(aR)}a\eta,
 \qquad \nu_R=-a\sinh(aR)x-\cosh(aR)\eta.
\]
Direct Lorentzian inner products give $\langle X_R,X_R\rangle=-a^{-2}$, $\langle\nu_R,\nu_R\rangle=1$, and $\langle X_R,\nu_R\rangle=0$. Also
\[
 dX_R=\cosh(aR)\,dx,\qquad d\nu_R=-a\sinh(aR)\,dx,
\]
so $A=a\tanh(aR)\Id$. The induced metric is $\cosh^2(aR)g_P$, which is complete. The map is an embedding because the formula is injective with a smooth inverse on its image.

For a horosphere use the upper half-space metric $\bar g=(az)^{-2}(dy^2+dz^2)$ on $z>0$. A slice $z=z_0$ with normal $\nu=az\partial_z$ satisfies $\bar\nabla_Y\nu=-aY$ for every tangential $Y$, as follows from $\bar\nabla_Y\partial_z=-Y/z$. Hence $A=a\Id$. Its induced metric is a constant multiple of the Euclidean metric, so this embedded hypersurface is complete. Formula \eqref{eq:umbilical-spectral} gives all three spectral ranges in the table.

\begin{proof}[Proof of Proposition~\ref{prop:sharp-boundary}]
Fix $\sigma<nc$ and set $\lambda=\sqrt{c-\sigma/n}>0$. In the flat and spherical cases choose the geodesic sphere with this value of $\lambda$. In the hyperbolic case, use a geodesic sphere if $\lambda>a$, a horosphere if $\lambda=a$, and an equidistant hypersurface if $0<\lambda<a$. The formulas above give the desired complete embedded model in every case, and \eqref{eq:umbilical-spectral} proves its equation. These examples exist for every admissible law as well. Together with Theorem~\ref{thm:nonlinear}, they show that $\sigma\ge nc$ is exactly the unconditional minimality range. This does not assert a classification of all subcritical solutions or an optimal non-CMC threshold.
\end{proof}

\section{Induced-metric linearization and model kernels}\label{sec:local}
This section does not use Theorem~\ref{thm:endpoint}. We compute the full induced-metric linearization and identify its joint kernel. Section~\ref{sec:nonlinear-local} then proves coercivity and the nonlinear estimate modulo ambient isometries. Throughout, the graph is global over the entire compact model, and its mean curvature stays away from zero.

\subsection{The defects and their induced-metric linearization}
Let $X_0:M_0\to N^{n+1}(c)$ be one of the models in Theorem~\ref{thm:local}, with constant trace $\tau_0\ne0$ and $S_0=|A_0|^2=nc-\sigma$. Its normal graph is
\[
 X_u(x)=\exp_{X_0(x)}(u(x)\nu_0(x)).
\]
All operators below are formed using the induced metric of $X_u$, then pulled back by $X_u$ to $M_0$. Use the differential of $X_u$ to identify tangential vectors with $TM_0$. Define the scalar and vector defects
\begin{align}
 \mathcal F_N(u)&=\Delta_u\Phi(\tau_u)
                 -(S_u-nc+\sigma)\Phi(\tau_u),\label{eq:normal-defect}\\
 \mathcal F_T(u)&=(2\Phi'(\tau_u)A_u+\Phi(\tau_u)\Id)\nabla_u\tau_u.
 \label{eq:tangent-defect}
\end{align}
Their vanishing is exactly \eqref{eq:spectral-main}. The norms in the estimates are taken with the fixed smooth metric on $M_0$; small graphs give uniformly equivalent pulled-back metrics.

For a normal speed $u\nu_0$, the first variations with our signs are
\begin{equation}\label{eq:variation}
 \dot g=-2u h_0,\qquad
 \dot A=\Hess u+u(A_0^2+c\Id),\qquad
 \psi:=\dot\tau=(\Delta+S_0+nc)u.
\end{equation}
To check the metric-dependent signs, use coordinates at a point with $g_{ij}=\delta_{ij}$ and $\Gamma^k_{ij}=0$, and write $X_i=dX(\partial_i)$. Differentiating the normal conditions and the tangent vectors gives
\[
 \dot\nu=-\nabla u,\qquad
 \bar\nabla_{\partial_t}X_j=u_j\nu-uA_0X_j,
 \qquad \dot g^{ij}=2u h_0^{ij}.
\]
Since $h_{ij}=\langle\bar\nabla_iX_j,\nu\rangle$, the curvature commutator yields
\begin{align*}
 \dot h_{ij}
 &=\langle\bar\nabla_i(u_j\nu-uA_0X_j),\nu\rangle
   +\langle\bar R(u\nu,X_i)X_j,\nu\rangle
   +\langle\bar\nabla_iX_j,-\nabla u\rangle\\
 &=(\Hess u)_{ij}-u(A_0^2)_{ij}+cu g_{ij}.
\end{align*}
Differentiating $A^i{}_j=g^{ik}h_{kj}$ adds $2uA_0^2$:
\[
 \dot A^i{}_j
 =2u(A_0^2)^i{}_j+(\Hess u)^i{}_j-u(A_0^2)^i{}_j+cu\delta^i_j.
\]
This proves \eqref{eq:variation}; its trace gives $\psi$. In particular,
\begin{equation}\label{eq:Svariation}
 \dot S=2\ip{A_0}{\Hess u}
          +2u\tr(A_0^3+cA_0).
\end{equation}
These are variations of the equation with the newly induced metric, not the Hessian of a functional with one frozen domain metric.

Let
\begin{equation}\label{eq:r-local}
 r_0=\frac{\tau_0\Phi'(\tau_0)}{\Phi(\tau_0)}
   =\frac{|\tau_0|F'(|\tau_0|)}{F(|\tau_0|)}>0.
\end{equation}
Write $\Phi_0=\Phi(\tau_0)$ and $\Phi_0'=\Phi'(\tau_0)$. Differentiating the normal defect before simplifying gives
\[
 \begin{aligned}
 D\mathcal F_N(0)u
   &=(\dot\Delta)\Phi_0+\Delta(\Phi_0'\psi)
     -\dot S\Phi_0-(S_0-nc+\sigma)\Phi_0'\psi\\
   &=\Phi_0'\Delta\psi-\Phi_0\dot S.
 \end{aligned}
\]
Here every scalar Laplacian annihilates the constant $\Phi_0$, so $(\dot\Delta)\Phi_0=0$, and $S_0-nc+\sigma=0$ by the model equation. Similarly, the derivative of the coefficient in \eqref{eq:tangent-defect} is multiplied by $\nabla\tau_0=0$. The metric variation in $\nabla_u\tau_0$ also vanishes because $d\tau_0=0$. Thus
\[
 D\mathcal F_T(0)u=(2\Phi_0'A_0+\Phi_0\Id)\nabla\psi.
\]
Dividing by the positive constant $\Phi_0/\tau_0$ gives
\begin{equation}\label{eq:linear-general}
 L_Nu=r_0\Delta \psi-\tau_0\dot S,
 \qquad L_Tu=(2r_0A_0+\tau_0\Id)\nabla \psi.
\end{equation}
Only the positive number $r_0$ enters these model operators. Higher derivatives of $F$ enter the nonlinear remainder, not the infinitesimal kernel.

\subsection{Round-sphere harmonics and completeness}
The only approximation theorem needed for the spectral decomposition is the following real form of Stone--Weierstrass \cite{Stone}.
\begin{standard}[Real Stone--Weierstrass theorem]\label{std:Stone-Weierstrass}
On a compact Hausdorff space $K$, a real subalgebra $\mathcal A\subset C(K,\R)$ is uniformly dense provided that $1\in\mathcal A$ and that for any distinct $x,y\in K$ some $f\in\mathcal A$ satisfies $f(x)\ne f(y)$.
\end{standard}
We apply this only to polynomial restrictions on round spheres and products of two round spheres. These sets are compact Hausdorff; the restrictions form an algebra, contain constants and separate points through the coordinate functions. Thus all of the stated hypotheses hold.

\begin{lemma}[Round-sphere spectrum]\label{lem:sphere-spectrum}
On $\Sph^m(R)$, $m\ge1$, the eigenvalues of $-\Delta$ are
$R^{-2}\ell(\ell+m-1)$, $\ell=0,1,\ldots$. The degree-$\ell$ eigenspace consists of the restrictions of homogeneous harmonic polynomials of degree $\ell$ in $\R^{m+1}$. These spaces give a complete orthogonal decomposition of $L^2$. Products of such harmonics give a complete joint decomposition for the two factor Laplacians on a product of round spheres.
\end{lemma}
\begin{proof}
First take $R=1$ and put $N=m+1$. For $x=r\omega$, the Euclidean Laplacian is
\[
 \Delta_{\R^N}=\partial_r^2+\frac{N-1}{r}\partial_r
                   +\frac1{r^2}\Delta_{\Sph^{N-1}}.
\]
If $H_\ell(r\omega)=r^\ell h(\omega)$ is harmonic, substitution gives
\[
 0=r^{\ell-2}\{\Delta_{\Sph^{N-1}}h+\ell(\ell+N-2)h\},
\]
which is the claimed eigenvalue. To prove completeness rather than merely produce eigenfunctions, direct differentiation gives
\[
 \Delta_{\R^N}(|x|^{2j}H_\ell)
 =2j(2\ell+N+2j-2)|x|^{2j-2}H_\ell\quad(j\ge1).
\]
The coefficient is nonzero for $N\ge2$. A polynomial of degree zero or one is already harmonic. For a homogeneous polynomial $P_q$ of degree $q\ge2$, apply the inductive decomposition to $\Delta P_q$, which has degree $q-2$. Each term $|x|^{2j}H_{q-2-2j}$ in that decomposition is the Laplacian of a nonzero scalar multiple of $|x|^{2j+2}H_{q-2-2j}$ by the displayed formula. Subtracting the sum of these primitives from $P_q$ leaves a harmonic homogeneous polynomial of degree $q$. Thus
\[
 P_q=H_q+|x|^2H_{q-2}+|x|^4H_{q-4}+\cdots.
\]
On the sphere this is a finite sum of harmonic restrictions. Standard theorem~\ref{std:Stone-Weierstrass} makes their span uniformly dense in continuous functions, hence dense in $L^2$. Different degrees are orthogonal: integration by parts gives
$(E_\ell-E_j)\langle h_\ell,h_j\rangle=0$, and the eigenvalues are strictly increasing. Completeness also excludes any further smooth eigenfunctions, by pairing a putative eigenfunction with every harmonic space. Rescaling the metric multiplies the eigenvalues by $R^{-2}$.

On a product, finite sums of products of polynomial restrictions again form a point-separating algebra containing constants. The same density argument and the decomposition on each factor give a complete joint harmonic basis. Orthogonality follows by integration on the two factors separately.
\end{proof}

\subsection{Compact nonminimal geodesic spheres in all three geometries}
The geodesic-sphere metric is $s_c(R)^2g_{\Sph^n(1)}$ by \eqref{eq:geodesic-radius-parameters}. In particular it is a round sphere of radius $K_0^{-1/2}$, including when $n=1$, where $K_0$ denotes the metric scale rather than an intrinsic sectional curvature. The three cases have
\[
\begin{array}{c|c|c}
 &\lambda\text{ (inward normal)}&K_0\\ \hline
c=0 & R^{-1} & R^{-2}\\
c>0 & \sqrt c\cot(\sqrt cR)&c/\sin^2(\sqrt cR)\\
c=-a^2<0 &a\coth(aR)&a^2/\sinh^2(aR).
\end{array}
\]
For the spherical row take the smaller-radius representative just described. Each model is paired with the fixed parameter $\sigma=n(c-\lambda^2)$; varying the radius generally varies that parameter.

Here $A_0=\lambda\Id$, $\lambda\ne0$, $\tau_0=n\lambda$ and $K_0=\lambda^2+c>0$. Equations \eqref{eq:variation}--\eqref{eq:Svariation} yield
\[
 \psi=(\Delta+nK_0)u,\qquad \dot S=2\lambda \psi,
\]
and hence
\begin{equation}\label{eq:sphere-linear}
 L_Nu=(r_0\Delta-2n\lambda^2)(\Delta+nK_0)u,
 \qquad L_Tu=\lambda(2r_0+n)\nabla(\Delta+nK_0)u.
\end{equation}
If both operators vanish, $\psi$ is constant by the tangential equation and then zero by the normal equation. Thus $(\Delta+nK_0)u=0$, and Lemma~\ref{lem:sphere-spectrum} identifies $u$ with a degree-one spherical harmonic. Conversely every such harmonic is in the common kernel.

Here is an explicit identification with ambient isometries. Let $o$ be the center and $W\in T_oN$. In Euclidean space take the translation field $K_W=W$. In the nonflat standard model take
\[
 K_W(z)=c\{\ip{o}{z}W-\ip{W}{z}o\}.
\]
This linear field is skew with respect to the ambient model scalar product, so it generates ambient isometries, and $K_W(o)=W$. Along a unit-speed radial geodesic $\gamma_\omega$ from $o$, the Killing equation gives
\[
 \frac{\dd}{\dd r}\ip{K_W(\gamma_\omega(r))}{\gamma_\omega'(r)}=0.
\]
Its constant value is $\ip W\omega$. The sphere normal is one of the two signs of $\gamma_\omega'$, so the normal component of $K_W$ is $\pm\ip W\omega$. As $W$ varies, these functions span all degree-one harmonics. The common kernel is therefore exactly the space of normal Killing variations and has dimension $n+1$.

\begin{corollary}[Quantitative rigidity of hyperbolic geodesic spheres]\label{cor:hyperbolic-local}
Fix $a,R>0$ and let $M_R$ be the geodesic sphere of radius $R$ in $\Hyp^{n+1}(-a^2)$, with the parameter in \eqref{eq:hyperbolic-sphere-intro}. Suppose $F$ is smooth with $F,F'>0$ near $na\coth(aR)$. Then the estimate and congruence conclusion of Theorem~\ref{thm:local} hold for sufficiently small smooth global $C^4$ normal graphs over $M_R$, with the induced metric of each graph. The exact infinitesimal kernel has dimension $n+1$ and consists of center-moving ambient Killing variations.
\end{corollary}
\begin{proof}
The radius calculation gives $\lambda=a\coth(aR)\ne0$ and $K_0=a^2/\sinh^2(aR)>0$, so all hypotheses of Theorem~\ref{thm:local}(a) hold. More explicitly, a spherical harmonic of degree $\ell$ has eigenvalue
\[
 E_\ell=K_0\ell(\ell+n-1)
\]
for $-\Delta$, and the normal multiplier in \eqref{eq:sphere-linear} is
\[
 (r_0E_\ell+2n\lambda^2)(E_\ell-nK_0).
\]
The first factor is strictly positive. Thus the normal kernel, and hence the joint kernel, consists exactly of degree-one harmonics. The Killing-field identification above applies with the Lorentzian scalar product. The common coercivity, slice and remainder argument below then proves the estimate and nonlinear conclusion, independently of the endpoint theorem.
\end{proof}
The constants here may depend on $a,R$ and on the law near the model amplitude. No uniform limit as $R\to\infty$ is asserted: $K_0\to0$, and the limiting horospherical model is noncompact. Nor does this result claim a compact proper biharmonic hypersurface in hyperbolic space, since its fixed spectral parameter is strictly less than $-2na^2$, not zero.

\subsection{The proper equal-radius Clifford models}
We first take $c=1$. Put
\[
 M_0=\Sph^a(1/\sqrt2)\times\Sph^b(1/\sqrt2),\quad
 a,b\ge1,\quad D=a-b\ne0,\quad n=a+b.
\]
For $X=(x,y)$ choose $\nu=(-x,y)$. Then $A_0$ is $+\Id$ on the first factor and $-\Id$ on the second, $\tau_0=D$, $S_0=n$ and $\sigma=0$. Let $\Delta_a,\Delta_b$ be the two scalar Laplacians and $\Delta=\Delta_a+\Delta_b$. Formula \eqref{eq:linear-general} becomes
\begin{align}
 L_Nu&=r_0\Delta(\Delta+2n)u
       -D\{2(\Delta_a-\Delta_b)u+4Du\},\label{eq:Clifford-LN}\\
 L_Tu&=(2r_0A_0+D\Id)\nabla(\Delta+2n)u.\label{eq:Clifford-LT}
\end{align}
Consider a tensor product of harmonic degrees $\ell,j$. Write
\[
 \alpha=2\ell(\ell+a-1),\quad
 \beta=2j(j+b-1),\quad E=\alpha+\beta.
\]
The normal multiplier and the squared tangential multiplier are
\begin{align}
 N_{r_0}(\alpha,\beta)&=r_0E(E-2n)+2D(\alpha-\beta)-4D^2,
 \label{eq:normal-multiplier}\\
 T_{r_0}(\alpha,\beta)&=(2n-E)^2
       \{(D+2r_0)^2\alpha+(D-2r_0)^2\beta\}.
 \label{eq:tangent-multiplier}
\end{align}
For a joint harmonic $u$, one has $\Delta_a u=-\alpha u$, $\Delta_bu=-\beta u$ and $\psi=(2n-E)u$. The tangential operator splits into the orthogonal factor components
\[
 L_Tu=(2n-E)\{(D+2r_0)\nabla_a u+(D-2r_0)\nabla_b u\}.
\]
Integration by parts gives
$\|\nabla_a u\|_2^2=\alpha\|u\|_2^2$ and
$\|\nabla_b u\|_2^2=\beta\|u\|_2^2$, proving \eqref{eq:tangent-multiplier}. More generally,
\[
 \langle\nabla_a u,\nabla_a v\rangle_{L^2}
       =\langle u,-\Delta_a v\rangle_{L^2},
\]
and likewise for the second factor. Thus orthogonal joint harmonics remain orthogonal for these gradient pairings, even if their total eigenvalues coincide. Together with Lemma~\ref{lem:sphere-spectrum}, this justifies determining the common kernel mode by mode.

If $D\ne\pm2r_0$ and $E\ne2n$, vanishing of $T_{r_0}$ forces $\alpha=\beta=0$, where $N_{r_0}=-4D^2\ne0$. If $E=2n$, vanishing of $N_{r_0}$ forces $\alpha-\beta=2D$, hence $\alpha=2a$, $\beta=2b$ and $(\ell,j)=(1,1)$. The exceptional values $D=\pm2r_0$ require a separate argument.

\begin{lemma}[The exceptional arithmetic obstruction]\label{lem:arithmetic}
Let $r_0>0$, let $a,b\ge1$ be integers with $a\ne b$, and set $n=a+b$ and $D=a-b$. For integers $\ell,j\ge0$, put
\[
 \alpha=2\ell(\ell+a-1),\qquad \beta=2j(j+b-1).
\]
On this discrete spectral set, the expressions in \eqref{eq:normal-multiplier} and \eqref{eq:tangent-multiplier} vanish simultaneously if and only if $(\ell,j)=(1,1)$, equivalently $(\alpha,\beta)=(2a,2b)$.
\end{lemma}
\begin{proof}
Only $D=-2r_0$ or $D=2r_0$ remains. By interchanging the factors we may suppose $D=-2r_0$, so $b=a+2r_0$. For $E\ne2n$, vanishing of the tangential multiplier requires $\beta=0$. Dividing the normal multiplier by $r_0$ gives
\begin{equation}\label{eq:exceptional-P}
 P(\alpha)=\alpha^2-(2n+4)\alpha-16r_0=0,
 \qquad\alpha=2\ell(\ell+a-1).
\end{equation}
The constant term is negative, so the two real roots have opposite signs; in particular there is exactly one positive root. At $\ell=2$ one has $\alpha=4(a+1)$ and $P(\alpha)=-16r_0(a+2)<0$. This value lies below the positive root. Since the allowed $\alpha$ increase with $\ell$, none of the degrees $0,1,2$ is a zero.

For $\ell\ge3$, use $2r_0=b-a$ and $n=a+b$ to rewrite \eqref{eq:exceptional-P} as
\[
 (2\alpha+8)b=\alpha^2-(2a+4)\alpha+8a.
\]
Polynomial division by $2(\alpha+4)$ gives
\begin{equation}\label{eq:integer-obstruction}
 b=\frac\alpha2-a-4+\frac{8(a+2)}{\alpha+4}.
\end{equation}
The first three terms are an integer. At $\ell=3$ the remaining fraction is strictly between $1$ and $4/3$, since $\alpha=6(a+2)$ and $a\ge1$. At $\ell\ge4$ it is strictly between $0$ and $1$, since $\alpha\ge8(a+3)$. It is never an integer, contradicting \eqref{eq:integer-obstruction}. The case $D=2r_0$ follows by symmetry. Together with the previously treated branch $E=2n$, this exhausts the simultaneous vanishing conditions.
\end{proof}
It follows that the common kernel $\KK$ consists of degree-$(1,1)$ harmonics. If $B:\R^{b+1}\to\R^{a+1}$ is any linear map, the ambient infinitesimal rotation
\[
 K_B(x,y)=(By,-B^{\mathsf T}x)
\]
has normal component $-2\ip{x}{By}$ on $M_0$. These functions span exactly $\KK$, of dimension $(a+1)(b+1)$. Rotations within either factor are tangent to $M_0$ and supply no additional normal component.

For $c>0$, compare the curvature-$c$ model with the unit-sphere model by multiplying its metric by $c$. The scaling rule of Proposition~\ref{prop:scaling}, with $b_0=\sqrt c$, transforms the equation with $\sigma=0$ into the unit-sphere equation with response $\widetilde F(t)=F(\sqrt c\,t)$. Its logarithmic derivative at $|D|$ is the same positive number as that of $F$ at $\sqrt c|D|$. Normal graphs and their smallness norms rescale by fixed constants. The kernel conclusion therefore holds for the model stated in Theorem~\ref{thm:local}.

\section{Coercivity and nonlinear rigidity}\label{sec:nonlinear-local}
The preceding section identifies the common kernel $\KK$ with the normal components of ambient Killing fields. We now complete the independent local proof in three steps: control all modes outside $\KK$, move each small graph into $\KK^\perp$, and absorb the nonlinear remainder.
\subsection{Spectral and geometric Sobolev norms}
A rotation-field identity gives a direct comparison between the spectral Sobolev norm and covariant derivatives.
\begin{lemma}[A rotation-field Sobolev identity]\label{lem:rotation-Sobolev}
On a round sphere of radius $R$, set
$Z_{ij}=R^{-1}(x_i\partial_j-x_j\partial_i)$, $i<j$.
On a product use all such fields on each factor. If $\mathcal Z$ denotes this finite family, then
\begin{equation}\label{eq:rotation-identities}
 \sum_{Z\in\mathcal Z} Z\otimes Z=g^{-1},\quad
 Z^*=-Z,\quad \Delta=\sum_{Z\in\mathcal Z}Z^2,\quad [Z,\Delta]=0.
\end{equation}
For smooth $u$ and every integer $k\ge0$,
\begin{equation}\label{eq:word-Sobolev}
 \sum_{Z_1,\ldots,Z_k\in\mathcal Z}
 \|Z_1\cdots Z_ku\|_2^2
       =\langle(-\Delta)^ku,u\rangle_{L^2}.
\end{equation}
In particular, define the spectral norm by
\[
 \|u\|_{H^4}^2:=\sum_E(1+E)^4\|u_E\|_2^2.
\]
There are constants $c_1,c_2>0$, independent of $u$, such that
\[
 c_1\|u\|_{H^4}^{\,2}
 \le\sum_{j=0}^4\|\nabla^ju\|_2^2
 \le c_2\|u\|_{H^4}^{\,2}.
\]
\end{lemma}
\begin{proof}
For a tangent vector $w$ at $x\in\Sph^m(R)$,
\[
 \sum_{i<j}\ip{Z_{ij}}{w}^2
 =\frac1{R^2}\sum_{i<j}(x_iw_j-x_jw_i)^2
 =\frac{|x|^2|w|^2-\ip{x}{w}^2}{R^2}=|w|^2.
\]
Polarization gives the tensor identity, and applying it to each factor gives the product version. Rotations preserve the metric and volume, proving skew-adjointness and commutation with $\Delta$. The tensor identity and divergence then give the sum-of-squares formula. For \eqref{eq:word-Sobolev}, the case $k=1$ is integration by parts. Each rotation field preserves every $-\Delta$ eigenspace because it commutes with $\Delta$. If $u_E$ lies in the eigenspace of eigenvalue $E$, and $W$ ranges over words of length $k$, then
\[
 \sum_{Z,W}\|ZWu_E\|_2^2
 =\sum_W\langle-\Delta Wu_E,Wu_E\rangle
 =E\sum_W\|Wu_E\|_2^2.
\]
Starting from length zero, induction gives $E^k\|u_E\|_2^2$. Distinct eigenspaces remain orthogonal under all words; Parseval and integration by parts give the stated formula for smooth $u$. Hence
\[
 \sum_{k=0}^4\binom4k
 \sum_{Z_1,\ldots,Z_k}\|Z_1\cdots Z_ku\|_2^2
       =\sum_E(1+E)^4\|u_E\|_2^2.
\]
Each word is a contraction of derivatives of $u$ of order at most its length, with bounded smooth coefficients. Conversely the first identity in \eqref{eq:rotation-identities} expresses every local coordinate derivative as a smooth linear combination of the $Z$'s. Iterating and using the product rule expresses all coordinate derivatives through order four as bounded-coefficient combinations of words. A finite atlas and the equivalence of coordinate and covariant derivatives prove both norm comparisons.
\end{proof}
\subsection{Coercivity of the joint operator}
For either compact model, let $\KK$ be the common kernel already computed and use the spectral norm in Lemma~\ref{lem:rotation-Sobolev}. On a Clifford harmonic mode the joint squared weight is exactly
\[
 W_{\mathrm{Cl}}(\alpha,\beta)
   =|N_{r_0}(\alpha,\beta)|^2+T_{r_0}(\alpha,\beta).
\]
The quantity $T_{r_0}$ already records the squared $L^2$ norm of the tangential operator on a unit-normalized mode; it is not squared again. On a geodesic sphere the corresponding weight, from \eqref{eq:sphere-linear}, is
\[
 W_{\mathrm{sph}}(E)
 = (r_0E+2n\lambda^2)^2(E-nK_0)^2
   +\lambda^2(2r_0+n)^2E(E-nK_0)^2.
\]
In either model the normal multiplier has the form
$N=r_0E^2+R$ with $|R|\le C(1+E)$; in the Clifford case use $|\alpha-\beta|\le E$. Choose $E_*\ge1$ so that $C(1+E)\le r_0E^2/2$ for $E\ge E_*$. Then
\[
 W\ge|N|^2\ge\tfrac14r_0^2E^4
              \ge\tfrac1{64}r_0^2(1+E)^4\qquad(E\ge E_*).
\]
There are only finitely many joint harmonic spaces with $E<E_*$. By the common-kernel calculations, the minimum
\[
 c_{\mathrm{low}}=\min_{\substack{E<E_*\\\text{mode outside }\KK}}
                       \frac{W}{(1+E)^4}
\]
is positive. If this finite set is empty, take $c_{\mathrm{low}}=1$;
the low-frequency assertion is then vacuous. Put $c_0=\min\{c_{\mathrm{low}},r_0^2/64\}$ and take $u\perp\KK$. Its harmonic expansion contains only modes outside $\KK$, so the scalar and factor-gradient orthogonality gives
\[
 \|L_Nu\|_2^2+\|L_Tu\|_2^2
     =\sum_{\text{modes}}W\|u_{\text{mode}}\|_2^2
     \ge c_0\sum_{\text{modes}}(1+E)^4\|u_{\text{mode}}\|_2^2.
\]
The last sum is $\|u\|_{H^4}^2$. Taking square roots gives
\begin{equation}\label{eq:coercivity}
 \norm{u}_{H^4}\le C_0\bigl(\norm{L_Nu}_{L^2}+\norm{L_Tu}_{L^2}\bigr),
 \qquad u\perp\KK.
\end{equation}
This argument does not assert that $L_N$ alone has no extra low-frequency zeros on every Clifford model; the tangential defect is essential to the all-parameter Clifford argument. On a geodesic sphere, by contrast, $r_0E+2n\lambda^2>0$, so the normal operator alone already has precisely the degree-one kernel.

\subsection{A finite-dimensional slice}
The required normal projection and signed height come from the following compact form of the tubular neighborhood theorem. We give the argument in the Riemannian ambient setting actually used here.
\begin{standard}[Uniform normal tubular neighborhood]\label{std:tube}
Let $M_0$ be a compact smooth embedded hypersurface without boundary in a smooth Riemannian manifold, and suppose that it has a smooth global unit normal $\nu_0$. For some $\rho_0>0$, the map
\[
 \mathscr E:M_0\times(-\rho_0,\rho_0)\longrightarrow N,
 \qquad \mathscr E(x,t)=\exp_x(t\nu_0(x)),
\]
is a diffeomorphism onto an open neighborhood of $M_0$. Its inverse has smooth components $(\pi,q_0)$, called normal projection and signed height.
\end{standard}
\begin{proof}
The exponential map is defined and smooth on a uniform neighborhood of the zero section over the compact set $M_0$. At $(x,0)$ its differential is
$(v,a)\mapsto v+a\nu_0(x)$, an isomorphism. Standard theorem~\ref{std:inverse} and compactness give a uniform smaller interval on which $\mathscr E$ is a local diffeomorphism. If no such interval gave injectivity, there would be distinct pairs $(x_j,t_j)$ and $(y_j,r_j)$ with $|t_j|+|r_j|\to0$ and equal images. Compactness gives subsequential limits $x,y\in M_0$. The geodesic lengths imply
$\dist_N(x_j,y_j)\le|t_j|+|r_j|$, so embeddedness gives $x=y$. Both pairs then approach $(x,0)$, contradicting the local injectivity there. Thus a uniform injective restriction exists. An injective local diffeomorphism has an open image and a smooth inverse.
\end{proof}
\begin{lemma}[Killing-orthogonal graph slice]\label{lem:slice}
For either compact model, every sufficiently small smooth normal graph $u$ can be moved by a small ambient isometry to a global normal graph $v$ with $v\perp\KK$ in $L^2(M_0)$ and
\[
 \|v\|_{C^4}\le C\|u\|_{C^4}.
\]
The construction uses no graph derivatives beyond order four.
\end{lemma}
\begin{proof}
Both compact models are embedded and have the displayed global normals, so Standard theorem~\ref{std:tube} applies. Here is the finite-dimensional construction of the slice. Let $k_1,\ldots,k_q$ be a basis of $\KK$ realized by normal components of ambient Killing fields, and let $I_z$ be the composition of their flows. In a tubular neighborhood write $\pi$ for the normal projection onto $M_0$ and $q_0$ for signed height. For a given graph put
\[
 p_z=\pi\circ I_z\circ X_u,\qquad
 h_z=q_0\circ I_z\circ X_u,\qquad
 v_z=h_z\circ p_z^{-1}.
\]
For $|z|+\|u\|_{C^1}$ small, $p_z$ is $C^1$-close to the identity. Choose a finite collection of coordinate balls, with slightly larger coordinate balls, covering $M_0$. In each larger ball the coordinate derivative of $p_z$ is uniformly close to the identity. Integrating along straight segments in smaller convex coordinate balls gives injectivity for sufficiently close pairs of points; a Lebesgue number of the finite cover makes this distance threshold uniform. Pairs farther apart cannot have the same image, since $\sup_x\dist(p_z(x),x)$ can be chosen smaller than half that threshold. Thus $p_z$ is globally injective. Its invertible differential also makes it a local diffeomorphism, so its image is open. The image is closed because $M_0$ is compact. Since each model $M_0$ is connected and the image is nonempty, the image is all of $M_0$. Hence $p_z$ is a global diffeomorphism, with its smooth inverse supplied locally by the ordinary inverse function theorem. The moments can be written without differentiating the inverse reparametrization:
\begin{equation}\label{eq:slice-moment}
 \mathcal B_j(z,u)=\int_{M_0}h_z(x)k_j(p_z(x))\operatorname{Jac}(p_z)(x)\dd\mu_0(x)
                =\int_{M_0}v_zk_j\dd\mu_0.
\end{equation}
They depend continuously on $(z,u)$ for $u$ in $C^1$, are differentiable in $z$, and
\[
 D_z\mathcal B(0,0)=G,
 \quad G_{ij}=\int_{M_0}k_ik_j\dd\mu_0,
 \quad G>0.
\]
To check the derivative, at $(z,u)=(0,0)$ one has $h_0=0$, $p_0=\Id$ and $\operatorname{Jac}(p_0)=1$. The variation of height is
$\partial_{z_i}h_z|_{(0,0)}=k_i$. All derivatives of the other two factors in \eqref{eq:slice-moment} are multiplied by $h_0$, so they vanish, leaving
$\partial_{z_i}\mathcal B_j(0,0)=\int k_i k_j$. The Gram matrix is positive definite because the $k_i$ are linearly independent. Continuity of these derivatives follows from the same integral, which uses only $u$ and its first derivatives. Uniformly on a small neighborhood,
$\|I-G^{-1}D_z\mathcal B\|\le1/2$ and
$|G^{-1}\mathcal B(0,u)|\le C\|u\|_{C^1}$.
Set $\Xi_u(z)=z-G^{-1}\mathcal B(z,u)$ and $R_u=2C\|u\|_{C^1}$. The first estimate makes $\Xi_u$ Lipschitz with constant at most $1/2$. For $|z|\le R_u$,
\[
 |\Xi_u(z)|\le|\Xi_u(0)|+\tfrac12|z|
       \le C\|u\|_{C^1}+\tfrac12R_u=R_u.
\]
The parameter ball is complete and is mapped into itself; after reducing $\|u\|_{C^1}$ it lies within the neighborhood where all estimates hold. Standard theorem~\ref{std:contraction} gives a fixed point, equivalently $\mathcal B(z,u)=0$. Thus $v_z\perp\KK$ and
\[
 |z|\le C\|u\|_{C^1},\qquad
 \|v_z\|_{C^4}\le C\|u\|_{C^4}.
\]
For the second estimate, the smooth tube maps give
\[
 \|h_z\|_{C^4}\le C(\|u\|_{C^4}+|z|).
\]
The inverse of $Dp_z$ is uniformly bounded. In coordinates, with $q=p_z^{-1}$,
\[
 Dq=(Dp_z\circ q)^{-1},\qquad
 D^2q=-(Dp_z\circ q)^{-1}(D^2p_z\circ q)[Dq,Dq].
\]
For completeness, fix a coordinate direction $\xi$ and put
$q_j=D^jq[\xi,\ldots,\xi]$, $p=p_z$; all derivatives of $p$ in the next
two identities are evaluated at $q(x)$. Differentiating $p\circ q=\Id$
three and four times gives
\begin{align}
 Dp\,q_3+3D^2p[q_1,q_2]+D^3p[q_1,q_1,q_1]&=0,
                 \label{eq:inverse-third}\\
 \begin{split}
 Dp\,q_4+4D^2p[q_1,q_3]+3D^2p[q_2,q_2]
 &+6D^3p[q_1,q_1,q_2]\\
 &+D^4p[q_1,q_1,q_1,q_1]=0.
 \end{split}\label{eq:inverse-fourth}
\end{align}
Mixed-direction identities follow by polarization of these symmetric
multilinear derivatives. At each order $j$, the sole term containing $D^jq$
is $(Dp\circ q)D^jq$; every other term contains derivatives of $p$ of order
at most $j$ and derivatives of $q$ of order strictly less than $j$.
The smooth tube maps give a uniform $C^4$ bound for $p$ when
$|z|+\|u\|_{C^4}$ is small, and $(Dp)^{-1}$ is uniformly bounded.
Solving the displayed identities successively therefore bounds $q$ in $C^4$
on the fixed finite atlas, without a fifth derivative of $p$ or $u$.
Every term in a derivative through order four of $h_z\circ q$ contains a
derivative of $h_z$ of order at most four multiplied by uniformly bounded
derivatives of $q$. Consequently
\[
 \|h_z\circ q\|_{C^4}\le C\|h_z\|_{C^4}
 \le C(\|u\|_{C^4}+|z|)\le C\|u\|_{C^4},
\]
which proves the claimed estimate. The construction uses graph derivatives only through order four and finite-dimensional differentiability in the isometry parameter. Ambient isometries preserve $F,c,\sigma$ and the equation.
\end{proof}

\subsection{Nonlinear remainder and the quantitative conclusion}
We first make the derivative count in the graph equation explicit. For a scalar
$u$, the notation $j^ku$ means the tuple consisting of $u$ and all its coordinate
derivatives through order $k$, taken in a fixed finite atlas on $M_0$.

\begin{lemma}[Exact normal-graph geometry and the fourth-order principal part]
\label{lem:graph-jets}
For either fixed compact model, let $\mathcal P_r$ denote parallel transport
along $r\mapsto\exp_{X_0(x)}(r\nu_0(x))$, and let $\nu_{\parallel}$ be the
unit tangent to this normal geodesic. Put
\[
 Q(r)=C_c(r)\Id-s_c(r)A_0,
 \qquad W_u=\sqrt{1+|Q(u)^{-1}\nabla_0u|_{g_0}^2}.
\]
For all sufficiently small $C^2$ graphs, $Q(u)$ is uniformly invertible and
\begin{align}
 dX_u(\xi)&=\mathcal P_u(Q(u)\xi)+du(\xi)\nu_{\parallel},
                   \label{eq:graph-differential}\\
 g_u(\xi,\eta)&=g_0(Q(u)\xi,Q(u)\eta)+du(\xi)du(\eta),
                   \label{eq:graph-metric}\\
 \nu_u&=W_u^{-1}\{\nu_{\parallel}
             -\mathcal P_u(Q(u)^{-1}\nabla_0u)\}.
                   \label{eq:graph-normal}
\end{align}
In each fixed coordinate chart there are smooth functions $B_{ij},b$ of
$(x,u,\partial u)$ such that
\begin{equation}\label{eq:graph-trace-jets}
 h^u_{ij}=W_u^{-1}\partial_{ij}u+B_{ij}(x,u,\partial u),
 \qquad
 \tau_u=\frac{g_u^{ij}}{W_u}\partial_{ij}u+b(x,u,\partial u).
\end{equation}
Repeated coordinate indices are summed. In particular,
\begin{equation}\label{eq:graph-fourth-principal}
 \mathcal F_N(u)=\frac{\Phi'(\tau_u)}{W_u}
                   g_u^{ij}g_u^{ab}\partial_{ijab}u
                   +\mathscr R_N(x,j^3u),
 \qquad
 \mathcal F_T(u)=\mathscr P_T(x,j^3u),
\end{equation}
where the remainder functions are smooth on a small jet neighborhood.
There are fixed constants $c_*,C_*,C>0$ such that
\begin{equation}\label{eq:graph-uniform-domain}
 c_*g_0\le g_u\le C_*g_0,
 \qquad \|\tau_u-\tau_0\|_{\infty}\le C\|u\|_{C^2},
 \qquad |\tau_u|\ge |\tau_0|/2
\end{equation}
after reducing the $C^2$ neighborhood. On a fixed sufficiently small
fourth-jet neighborhood, both defects are smooth functions of $j^4u$ with
uniform bounds on their second derivatives in the jet variables.
\end{lemma}
\begin{proof}
Define the smooth tubular map
$\Xi(x,r)=\exp_{X_0(x)}(r\nu_0(x))$. At fixed $x$, the variation in an initial
direction $\xi\in T_xM_0$ is a Jacobi field satisfying
\[
 \bar\nabla_r^2J_\xi+cJ_\xi=0,
 \qquad J_\xi(0)=dX_0(\xi),\qquad
 \bar\nabla_rJ_\xi(0)=-dX_0(A_0\xi).
\]
The initial field and its derivative are orthogonal to the normal geodesic.
In parallel identification the unique solution is
$J_\xi(r)=\mathcal P_r(Q(r)\xi)$, because $C_c'=-cs_c$ and $s_c'=C_c$.
Applying the chain rule to $X_u(x)=\Xi(x,u(x))$ proves
\eqref{eq:graph-differential}. Its two terms are orthogonal, giving
\eqref{eq:graph-metric}.
Since $Q(r)$ is self-adjoint, the numerator on the right of
\eqref{eq:graph-normal} is orthogonal to $dX_u(\xi)$:
\[
 du(\xi)-g_0(Q(u)^{-1}\nabla_0u,Q(u)\xi)=0.
\]
Its length is $W_u$, and it equals $\nu_0$ at $u=0$, proving the claimed
choice of normal.

In fixed coordinates the second derivative of $\Xi(x,u(x))$, interpreted
covariantly in the ambient manifold, has the form
\[
 \bar\nabla_i\partial_jX_u
     =\partial_{ij}u\,\nu_{\parallel}
       +\mathcal B_{ij}(x,u,\partial u).
\]
Indeed, the other terms contain only $\Xi_{ij}$, $\Xi_{ir}\partial_ju$,
$\Xi_{rj}\partial_iu$, $\Xi_{rr}\partial_iu\partial_ju$, and the ambient
connection applied to the first derivatives; all are smooth in the indicated
variables. Taking the inner product with \eqref{eq:graph-normal} gives the
first formula in \eqref{eq:graph-trace-jets}, since
$\ip{\nu_{\parallel}}{\nu_u}=1/W_u$. Tracing with $g_u^{-1}$ gives the
second formula. Background coordinate-connection terms are included in
$B_{ij}$ and $b$.

For the principal part, write $a^{ab}=g_u^{ab}/W_u$, so
$\tau_u=a^{ab}(x,u,\partial u)\partial_{ab}u+b(x,u,\partial u)$.
The coefficient $a^{ab}$ uses only the first graph jet. Hence
\[
 \partial_{ij}\tau_u
   =a^{ab}\partial_{ijab}u+\mathscr B_{ij}(x,j^3u).
\]
The Christoffel symbols of $g_u$ use at most the second graph jet. Consequently
\[
 \Delta_u\tau_u=g_u^{ij}a^{ab}\partial_{ijab}u
                       +\mathscr B(x,j^3u).
\]
Now use
\[
 \Delta_u\Phi(\tau_u)
 =\Phi'(\tau_u)\Delta_u\tau_u
   +\Phi''(\tau_u)|\nabla_u\tau_u|_{g_u}^2.
\]
The last term uses at most third derivatives, $S_u$ uses at most second
derivatives, and the pulled-back tangential endomorphism $A_u$ is smooth in
the second jet. This proves \eqref{eq:graph-fourth-principal}. Notice that its
fourth-order coefficient may itself depend on $j^2u$ through $\tau_u$;
we do not claim that it depends only on the first jet.

Finally $Q(0)=\Id$, and compactness gives uniform invertibility for small
$\|u\|_{\infty}$. The exact metric formula gives uniform equivalence for
small $\|u\|_{C^1}$. The trace formula and the ordinary mean-value estimate
in its finite-dimensional jet variables give
$\|\tau_u-\tau_0\|_\infty\le C\|u\|_{C^2}$.
Choose this difference smaller than $|\tau_0|/2$ and, if necessary, smaller
still so that $|\tau_u|$ remains in a fixed compact subinterval of the given
response-law domain. The sign of $\tau_u$ is then the sign of $\tau_0$.
All inverses and denominators above are uniformly controlled. On relatively
compact coordinate charts and a fixed small closed jet ball, smoothness
therefore supplies a uniform bound for the second jet derivatives of the
defects. A finite atlas makes the bound uniform on $M_0$.
\end{proof}

\begin{lemma}[Fourth-order nonlinear remainder]\label{lem:remainder}
For all sufficiently small smooth $C^4$ normal graphs over either fixed model, one has
\begin{equation}\label{eq:remainder}
 \norm{\mathcal F(u)-\frac{\Phi(\tau_0)}{\tau_0}L u}_{L^2\oplus L^2}
 \le C_1\norm u_{C^4}\norm u_{H^4}.
\end{equation}
The constants use only the fixed model and the response law near its nonzero amplitude.
\end{lemma}
\begin{proof}
Lemma~\ref{lem:graph-jets} gives explicit formulas for all geometric quantities,
ensures immersion and separation from zero mean curvature, and represents the
defects as smooth functions $\mathscr P(x,j^4u)$ on a fixed small jet
neighborhood. It also supplies a uniform bound for the second derivatives of
$\mathscr P$ with respect to the jet variables. Set $a_0=\Phi(\tau_0)/\tau_0$. For a jet vector $\eta=j^4u(x)$, Taylor's formula in the finite-dimensional jet variables is
\[
 \mathscr P(x,\eta)-\mathscr P(x,0)-D_\eta\mathscr P(x,0)\eta
 =\int_0^1(1-t)D_\eta^2\mathscr P(x,t\eta)[\eta,\eta]\dd t.
\]
The integrand is uniformly bounded by $C|\eta|^2$. Since $\mathcal F(0)=0$ and $D\mathcal F(0)=a_0L$, this gives the pointwise remainder estimate $|\mathcal F(u)-a_0Lu|\le C|j^4u|^2$. Integrating in the fixed background metric and using Lemma~\ref{lem:rotation-Sobolev},
\[
 \big\||j^4u|^2\big\|_{L^2}
 \le\|j^4u\|_{L^\infty}\|j^4u\|_{L^2}
 \le C\|u\|_{C^4}\|u\|_{H^4}.
\]
A finite atlas and a partition of unity combine these local estimates
into the stated global estimate.
For example, the quadratic term $\Phi''(\tau_u)|\nabla_u\tau_u|^2$ contains products of third graph derivatives; one factor is controlled in $L^\infty$ and the other in $L^2$. The independent $C^4$ smallness, not an embedding $H^4\hookrightarrow C^4$, supplies the required pointwise bound in every dimension.
\end{proof}

\begin{proof}[Proof of Theorem~\ref{thm:local}]
Use Lemma~\ref{lem:slice} to obtain $v\perp\KK$. Put $a_0=\Phi(\tau_0)/\tau_0>0$, and write $\mathcal H=L^2(M_0)\oplus L^2(TM_0)$ for the scalar--vector Hilbert sum with the fixed background metric. Enlarging the constant in \eqref{eq:coercivity} if necessary, let $C_L$ satisfy $\|v\|_{H^4}\le C_L\|Lv\|_{\mathcal H}$ for $v\perp\KK$. If $C_R$ is a remainder constant in \eqref{eq:remainder}, then
\[
 \|v\|_{H^4}
 \le \frac{C_L}{a_0}\|\mathcal F(v)\|_{\mathcal H}
   +\frac{C_L C_R}{a_0}\|v\|_{C^4}\|v\|_{H^4}.
\]
Choose $\|v\|_{C^4}\le a_0/(2C_LC_R)$, taking $C_R>0$. The slice estimate ensures this after reducing the original graph neighborhood. Absorption gives
\[
 \|v\|_{H^4}\le \frac{2C_L}{a_0}\|\mathcal F(v)\|_{\mathcal H}
 \le \frac{2C_L}{a_0}\bigl(\|\mathcal F_N(v)\|_{L^2}
                         +\|\mathcal F_T(v)\|_{L^2}\bigr).
\]
This proves \eqref{eq:quantitative-main}. Since the isometry preserves the equation, an exact solution gives $\mathcal F(v)=0$ and hence $v=0$. For parameter uniformity, let $F(t)=t^{p-1}$ and fix
$I=[p_-,p_+]\Subset(1,\infty)$. Set $r_-=p_--1>0$.
Uniformly for $p\in I$, the normal multiplier has the form
\[
 N_p=(p-1)E^2+R_p,\qquad |R_p|\le C_I(1+E).
\]
Choose one $E_*\ge1$ such that
$C_I(1+E)\le r_-E^2/2$ for $E\ge E_*$. Then the joint squared weight
satisfies $W_p\ge r_-^2E^4/4$ for every such $E$ and every $p\in I$.
Let $\mathscr M_*$ be the finite collection of joint modes with $E<E_*$
outside the fixed Killing kernel. The kernel computations, including
Lemma~\ref{lem:arithmetic} at the exceptional Clifford parameters, give
$W_m(p)>0$ for every $m\in\mathscr M_*$ and every $p\in I$. Compactness
and continuity therefore imply
\[
 \min_{\substack{p\in I\\m\in\mathscr M_*}}
       \frac{W_m(p)}{(1+E_m)^4}>0
\]
when $\mathscr M_*$ is nonempty; when it is empty, only the high-frequency
bound is needed. Continuity alone would not give strict positivity without
the all-parameter kernel calculation. Thus $C_L$ is uniform. Also
$\inf_{p\in I}a_0(p)=\inf_{p\in I}|\tau_0|^{p-2}>0$.
Choose a fixed compact amplitude interval about $|\tau_0|$ away from zero.
The finitely many derivatives of $t^{p-1}$ needed in the jet remainder
are uniformly bounded on the product of that interval with $I$.
Lemma~\ref{lem:graph-jets} then gives a uniform remainder constant, while
the slice depends only on the fixed model. The absorption inequality therefore
uses one common $C^4$ neighborhood and one common final constant.
\end{proof}

\section{Scope and limitations}\label{sec:scope}
\subsection{Geometric scope}
The gradient-line reduction starts from a nonzero-mean-curvature region on which $d\tau\ne0$. After restriction to a constant-multiplicity neighborhood, that region supplies an open interval of the necessary ODE system. Analytic continuation concerns the state of this system; it need not produce an extended immersion. A nonzero CMC region is treated separately by $|A|^2=\Lambda$. These distinctions explain why the nonexistence theorem requires neither completeness of the hypersurface nor a global principal-coordinate net.

Constant sectional curvature enters both the tangential equation and Codazzi. For a general ambient manifold, the ordinary tangential equation is
\[
 (2A+\tau\Id)\nabla\tau
   =2\tau(\Ric^N(\nu))^\top,
\]
and mixed ambient curvature also occurs in Codazzi. A sign condition on sectional curvature alone does not remove those terms. The counterexamples of Ou--Tang \cite{OT} preclude the corresponding unrestricted nonpositively curved extension. Higher codimension introduces a varying mean-curvature direction in the normal bundle, beyond the single-shape-operator reduction used here.

The nonexistence argument is local, so it applies in space-form quotients and ambient open subsets. For a smooth hypersurface with boundary, the argument applies in the interior. Minimality and non-strict bounds extend to the boundary by continuity; the strict gradient-point statements are asserted only at interior points. The graph theorem is formulated instead on an entire compact model in a simply connected space form. In particular, the hyperbolic part of the nonlinear spectral theory is not restricted to the ordinary biharmonic law or to $\sigma=0$, while its compact-model rigidity statement is restricted to geodesic spheres with $\sigma<-2na^2$.

\subsection{Response laws and variational conventions}
The hypotheses of Definition~\ref{def:admissible} serve distinct purposes. The strict inequality $R_F-tR_F'>0$ makes $k_F$ invertible, real analyticity gives persistence of the normal constraint, and the elasticity bounds give the two-sided trace comparison used at the endpoint. Thus positivity of $F$ and $F'$ alone is not the hypothesis of the nonexistence theorem. The compact-model theorem uses only local smoothness and $F,F'>0$ near $|\tau_0|>0$; it has an independent spectral and nonlinear proof.

The power functional is based on the ordinary tension field: in the convention $E_{q,r}=r^{-1}\int|\tau_q|^r$ of \cite{TCL}, it is $E_{2,p}$. Changing the tension exponent changes the equation. Likewise, a prescribed weight on the domain gives a different problem from a constitutive response depending on the mean curvature. When deriving the map Euler equation the domain metric is fixed, whereas each immersion tested against the resulting equation carries its own induced metric.

\subsection{Thresholds and compact-model neighborhoods}
Proposition~\ref{prop:sharp-boundary} establishes sharpness only of the unconditional spectral minimality range. The mean-curvature bounds in Theorems~\ref{thm:power} and \ref{thm:nonlinear} are sufficient one-point CMC criteria; they do not establish optimal curvature thresholds, imply that every spherical or subcritical hyperbolic solution is CMC, or give a global classification. The mixture-width condition is sufficient, and no necessity is claimed.

Theorem~\ref{thm:local} concerns small global normal graphs over the stated nonminimal geodesic spheres and equal-radius Clifford models, modulo ambient isometries. The independent $C^4$ smallness assumption is essential to the remainder estimate. Minimal Clifford products are outside its hypotheses because $\tau_0=0$; arbitrary unequal radii, large graphs, local patches and general quotient models are also outside the stated conclusion. Hyperbolic horospheres and equidistant hypersurfaces, although used as explicit spectral examples, are noncompact and are not covered by the compact harmonic decomposition or the resulting graph estimate. No fixed-domain stability or index interpretation of the induced-metric joint linearization is used.

\appendix

\section{Elementary analytic tools and a constraint warning}\label{app:quadratic}
\subsection{The contraction principle}
\begin{standard}[Banach contraction principle]\label{std:contraction}
Let $(\mathcal X,d_{\mathcal X})$ be a nonempty complete metric space and let $P:\mathcal X\to\mathcal X$ satisfy
$d_{\mathcal X}(Px,Py)\le q\,d_{\mathcal X}(x,y)$ for a constant $0\le q<1$.
The iteration $x_{m+1}=Px_m$ converges, from any initial $x_0$, to the unique fixed point $x_*$. Moreover,
\[
 d_{\mathcal X}(x_m,x_*)\le\frac{q^m}{1-q}d_{\mathcal X}(x_1,x_0).
\]
\end{standard}
\begin{proof}
Successive distances are bounded by $q^m d_{\mathcal X}(x_1,x_0)$. Summing their geometric tail makes $(x_m)$ Cauchy and proves the displayed estimate. Completeness supplies $x_*$, and continuity of $P$ gives $Px_*=x_*$. Two fixed points have distance at most $q$ times their distance, so they coincide. This is the usual contraction argument; compare \cite[Theorem~2.1]{Teschl}.
\end{proof}

\subsection{Proof of the analytic continuation lemma}\label{app:odeproof}
\begin{proof}[Proof of Lemma~\ref{lem:analytic}]
\emph{Existence and uniqueness.}
On a closed state ball inside $\Omega$, let $M$ bound $|\mathcal V|$ and $L$ bound $\|D\mathcal V\|$. Choose a time length $h$ with $hM$ at most the ball radius and $hL<1$. The map
\[
 z\longmapsto z_0+\int_0^t\mathcal V(z(u))\dd u
\]
is a contraction of the corresponding closed ball in the complete space of continuous paths with the uniform norm. The path ball is closed in that complete space, and the integral map both preserves it and has Lipschitz constant at most $hL<1$. Standard theorem~\ref{std:contraction} gives its unique fixed point; differentiating the integral equation gives the ODE. Uniqueness on overlaps glues local solutions to a unique maximal interval. The same argument proves the nonautonomous uniqueness used in Lemma~\ref{lem:matrix} whenever the right side is continuous in time and locally uniformly Lipschitz in the state.

\emph{Analyticity.}
Use a real coefficient majorant. After translation write
\[
 \mathcal V_i(z_0+w)=\sum_{\alpha\in\mathbb N^m}c_{i,\alpha}w^\alpha,
 \qquad
 \sum_{|\alpha|=q}|c_{i,\alpha}|\le C R^{-q}\quad(q\ge0).
\]
Such $C,R>0$ exist by absolute convergence on a smaller polydisc. The coefficients of the formal solution $w(t)=\sum_{j\ge1}a_jt^j$ are uniquely determined by
\[
 (j+1)a_{j+1,i}=[t^j]\mathcal V_i(z_0+w(t)).
\]
Compare, coefficient by coefficient, with the nonnegative series
\[
 W'=\frac{C}{1-W/R},\quad W(0)=0,
 \qquad W(t)=R\left(1-\sqrt{1-2Ct/R}\right).
\]
The first coefficient is dominated. If the comparison holds through degree $j$, then every product $w^\alpha$ through degree $j$ is dominated by $W^{|\alpha|}$ after absolute values; summing the coefficient bound gives domination of degree $j+1$. Thus every component of $w$ converges absolutely for $|t|<R/(2C)$. Termwise differentiation and multiplication are valid on smaller intervals, and the sum solves the ODE. Uniqueness identifies it with the real solution. Repeating at each interior time proves analyticity.

\emph{Persistence of an identity.}
Write $f(s)$ for the analytic scalar function along the trajectory and $I$ for its connected maximal time interval. Define
\[
 Z=\{s\in I:f\text{ vanishes identically on a neighborhood of }s\}.
\]
The initial open interval makes $Z$ nonempty, and $Z$ is open by definition. If $s_j\in Z$ tends to $s\in I$, then $f^{(m)}(s_j)=0$ for every $m\ge0$. Continuity of each derivative gives $f^{(m)}(s)=0$ for every $m$. Analyticity at $s$ then makes $f$ zero on a neighborhood of $s$, so $Z$ is also closed in $I$. Connectedness gives $Z=I$. The hypothesis is identical vanishing on an open time interval.

\emph{Continuation.}
If the state stays in a compact subset at a finite endpoint, its velocity is bounded. The trajectory is Lipschitz and has a limiting state still in $\Omega$. Local existence at that state and uniqueness extend the solution. These arguments are the usual basic continuation principles; see also \cite{Teschl}.
\end{proof}
\subsection{A real coefficient proof of quadratic ODE analyticity}
For the power-law state field, analyticity can be seen without a complex-time argument. At an arbitrary state write
\[
 \mathcal V(z_0+w)=b+Lw+B(w,w),
\]
where $B$ is a bounded bilinear map in the finite-dimensional state norm. Choose
\[
 A\ge\max\{1,|b|,\|L\|/2,\|B\|\}.
\]
The formal series $w(t)=\sum_{m\ge1}w_mt^m$ is determined recursively by
\[
 w_1=b,\qquad
 (m+1)w_{m+1}=Lw_m+\sum_{i=1}^{m-1}B(w_i,w_{m-i}).
\]
If $|w_i|\le A^i$ through $i=m$, the right side has norm at most
$2A^{m+1}+(m-1)A^{m+1}=(m+1)A^{m+1}$. Induction gives $|w_m|\le A^m$ for every $m$. The series and its derivative converge on each compact subinterval of $|t|<A^{-1}$; substitution is justified by absolute convergence. It solves the original ODE and agrees with its unique real solution. This proves analyticity for every interior initial state of the quadratic system. The more general analytic law uses Lemma~\ref{lem:analytic}.

\subsection{A single-state constraint counterexample}
In the ordinary Euclidean case, set $d=5$, $k=-\frac13\sum_i\lambda_i$, and
\[
 (\lambda_1,\ldots,\lambda_5)=(-3,0,0,0,0),\qquad
 (\mu_1,\ldots,\mu_5)=(-3\sqrt{14}/4,0,0,0,0).
\]
Let $\mathcal E=D^2k-\mathsf TDk-k(k^2+\mathsf Q)$. The subspace with $\lambda_i=\mu_i=0$ for $i\ge2$ is invariant under the pair equations when $c=0$. On it, write $L=\lambda_1$ and $M=\mu_1$. Then
\[
 k=-L/3,\qquad DL=\tfrac43LM,\qquad DM=M^2-L^2/3,
\]
so differentiation gives the polynomial identities
\[
 \mathcal E=\frac{2L(7L^2-8M^2)}{27},\qquad
 D\mathcal E=\frac{40LM(5L^2-4M^2)}{81}.
\]
At the displayed initial state these yield
\[
 k=1,\qquad \mathcal E=0,\qquad D\mathcal E=15\sqrt{14}\ne0.
\]
Thus a single initial zero of the normal residual is insufficient. Throughout the paper the original open time interval supplies the analytic identity. Equation \eqref{eq:power-residual} is $k\mathcal E$ when $r=1$, so the same example also warns against a single-state use of that equation.

\section*{Acknowledgment of AI assistance}
The author used ChatGPT in this work. The author takes full responsibility for the content of the paper.

\end{document}